\documentclass[12pt]{article}

\usepackage{graphicx}
\usepackage{multicol,multirow}
\usepackage{amsmath,amssymb,amsfonts}
\usepackage{mathrsfs}
\usepackage{rotating}
\usepackage{appendix}
\usepackage[numbers]{natbib}
\usepackage{amsthm}
\usepackage{lipsum}

\usepackage{graphicx}
\usepackage{ifpdf}
\usepackage{color}
\allowdisplaybreaks
\DeclareGraphicsExtensions{.pdf,.png,.jpg,.jpeg,.mps,.eps}

\usepackage{hyperref}

\def\R{\mathbb R}

\def\Z{\mathbb Z}
\def\G{\mathsf{G}}

\def\pr{{\rm pr }}
\def\int{{\rm int\,}}
\def\P{\mathcal P}

\def\eps{\varepsilon}
\def\PSL{{\rm PSL}}
\def\diam{{\rm diam\,}}

\def\T{{\mathcal T}}

\def\X{{\mathfrak X}}
\def\H{\mathbb H}

\def\X{\mathcal X}

\def\P{{\mathcal P}}
\newcommand{\SL}{{\rm SL}}
\newcommand{\C}{{\mathscr C}}

\newtheorem{theorem}{Theorem}[section]

\newtheorem{definition}[theorem]{Definition}

\newtheorem{proposition}[theorem]{Proposition}

\newtheorem{remark}[theorem]{Remark}

\newtheorem{corollary}[theorem]{Corollary}
\newtheorem{lemma}[theorem]{Lemma}

\usepackage{hyperref}
\usepackage{mathrsfs}

\begin{document} \vskip -2cm
			\title{Markov partitions for the geodesic flow on compact Riemann surfaces of constant negative curvature}
	\author{{\sc Huynh M. Hien}\\
	Department of Mathematics and Statistics,\\
		Quy Nhon University,\\
		 170 An Duong Vuong, Quy Nhon, Vietnam,\\
		e-mail: huynhminhhien@qnu.edu.vn}
	\date{}
	\maketitle 
	\begin{abstract}
		It is well-known that  hyperbolic flows admit  Markov partitions of arbitrarily small size. However, the constructions of Markov partitions for general hyperbolic flows are very abstract and not easy to understand. To establish a more detailed understanding of Markov partitions, in this paper we consider the geodesic flow on Riemann surfaces of constant negative curvature. We provide a rigorous construction of Markov partitions for this hyperbolic flow with explicit forms of rectangles and local cross sections. The local product structure is also calculated in detail.  
	\end{abstract}
{\bf Keywords:} Markov partitions; Symbolic dynamics; Geodesic flows; Constant negative curvatures
	
	\medskip
	\noindent	{\bf 2020 MSC: } {37B10,  37D40, 37D20,  57K32} 

\tableofcontents

%\tableofcontents
\section{Introduction} Symbolic dynamics has had a great history development and is a very useful method to 
study general dynamical systems. Instead of working on general dynamical systems, one can consider respective symbolic systems via symbolic dynamics.
The symbolic dynamics of a dynamical system is constructed from Markov partitions, which have been attracting a lot of mathematicians.  In 1967, a Markov partition for hyperbolic diffeomorphisms on 2-torus was constructed by Adler and Weiss in \cite{AW}.
Then Sinai \cite{Sinai,Sinai1} used successive approximations to construct a Markov partition for arbitrary $C$-diffeomorphisms. 
Bowen \cite{bo-diff} 
used Sinai's method to give a construction of Markov partitions for Smale's Axiom A diffeomorphisms with the help of Smale's Spectral Decomposition Theorem in \cite{smale}. In the case
of $C$-flows on three-dimensional manifolds, a construction of Markov partitions
was given by Ratner \cite{Ratner0}. The author also introduced a Markov partition for transitive Anosov flows (so-called $C$-flows) on $n$-dimensional manifolds \cite{Ratner}. In 1973, Bowen  modified and generalized the construction in \cite{bo-diff} to have a Markov partition for  
$C^1$-hyperbolic flows in \cite{bo-symb}, which has become a classic reference.
Pollicott \cite{Po87} then constructed symbolic dynamics for Smale flows, which is  a class of continuous flows on metric spaces provided a local product structure. 
The result generalizes Bowen's construction of symbolic dynamics for $C^1$-hyperbolic flows in \cite{bo-symb}.
The problem is that all the constructions of Markov partitions mentioned above are very abstract and not easy to understand.

Pollicott and Sharp have  found symbolic dynamics very useful in counting closed orbits for hyperbolic flows \cite{PS2}  and
presenting asymptotic estimates
for pairs of closed geodesics whose length differences lie in a prescribed
family of shrinking intervals \cite{PS1}; see also \cite{Po95,PS98} for other applications. Under supervision of Knieper, Bieder in his PhD thesis \cite{bieder} used symbolic dynamics to construct partner orbits for hyperbolic flows. 
The purpose of construction of symbolic dynamics is to prove that a hyperbolic flow is semi-conjugated to a hyperbolic symbolic flow,
 in which the symbolic dynamics must be constructed from a Markov partition. A Markov partition is family of rectangles satisfying the Markov property. 
 Then one can associate to a hyperbolic flow a mixing subshift of finite type
$\sigma: \Lambda\rightarrow\Lambda $
and a  H\"older continuous function 
$r : \Lambda\rightarrow\R$ such that, with at
most a finite number of exceptions, the prime periodic orbit 
$\{x, \sigma x, . . . , \sigma^{k-1}x\}$ corresponds to the prime periodic orbit
$\gamma$ whose word length and length are given by
$|\gamma|=k$ and $l_\gamma=r^k(x)=r(x)+r(\sigma x)+\cdots+r(\sigma^{k-1}x)$, respectively;
where $\Lambda=\{x =(x_n)_{n=-\infty}^\infty : A(x_n,x_{n+1})=1,\forall n\in {\mathbb Z}\}$,
$A$ is the corresponding adjacency matrix with entries 0 or 1 of the Markov partition
and $x_n$ are the symbols of rectangles in the Markov partition. 
Thus instead of working on  the original hyperbolic flow,  this viewpoint offers several advantages.

The main tool for the construction of the hyperbolic symbolic flow is the existence of a Markov partition for a basic set.   However, for general hyperbolic flows,
Markov partitions are not explicit and their constructions are not easy to understand as mentioned above. For instance, there are several results in \cite{bo-symb} which need to be carefully verified. 
To establish  a more detailed understanding of Markov partitions, in this paper we consider a concrete hyperbolic dynamical system, namely the geodesic flow on compact Riemann surfaces of constant negative curvature. We introduce explicit forms of rectangles as well as local cross sections. This leads to a more explicit and intuitive Markov partition for the system. Coordinalization of Poincar\'e sections helps us calculate the local product structure in detail and prove the existence
of a pre-Markov partition.
Especially,  we even could somewhat simplify \cite{bo-symb}, in that we do not need several of the lemmas in this work.  In addition, all important results in \cite{bo-symb} in relevant to Markov partitions are rigorously verified.

The paper is organized as follows. In the next section, we give an introduction to the theory of the geodesic flow on compact factors of the hyperbolic plane with auxiliary results that will be used in this paper. Section \ref{lpss} studies local product structure of the flow with specific calculations. Section \ref{lcssection} presents explicit forms of local cross sections and rectangles.  Expansivity of the flow is studied in Section \ref{expansection}. The final section provides a rigorous construction of Markov partitions for the flow.

\section{The geodesic flow on compact factors of the hyperbolic plane}
We consider the geodesic flow on compact Riemann surfaces of constant negative curvature. 
It is well-known that any compact orientable surface 
with a metric of constant negative curvature is isometric to a factor $\Gamma\backslash \H^2$, 
where $\H^2=\{z=x+iy\in {\mathbb C}:\, y>0\}$ is the hyperbolic plane endowed  
with the hyperbolic metric $ds^2=\frac{dx^2+dy^2}{y^2}$ and
$\Gamma $ is a discrete subgroup of the projective Lie group $\PSL(2,\R)=\SL(2,\R)/\{\pm E_2\}$; 
here $\SL(2,\R)$ is the group of all real $2\times 2$ matrices with unity determinant, 
and $E_2$ denotes the unit matrix.  The group $\PSL(2,\R)$  acts transitively on $\H^2$ by 
M\"obius transformations
$z\mapsto \frac{az+b}{cz+d}$.
If the action is free (of fixed points), then the factor $\Gamma\backslash\H^2$  has a Riemann surface structure.
Such a surface is a closed Riemann surface of genus at least $2$ 
and has the hyperbolic plane $\H^2$ as the universal covering. 
The geodesic flow $(\varphi_t^\X)_{t\in \R}$ on the unit tangent bundle $\X=T^1(\Gamma\backslash\H^2)$
goes along the unit speed geodesics on $\Gamma\backslash\H^2$.

On the other hand, the unit tangent bundle $T^1(\Gamma\backslash\H ^2)$
is isometric to the quotient space 
$\Gamma\backslash \PSL(2,\R)=\{\Gamma g,g\in\PSL(2,\R)\}$, 
which is the system of right co-sets of $\Gamma$ in $\PSL(2,\R)$, by an isometry
$\Xi$.
Then  the geodesic flow $(\varphi_t^\X)_{t\in\R}$ can be equivalently described as the natural 
`quotient flow'
\begin{equation}\label{varphix}\varphi^X_t(\Gamma g)=\Gamma g a_t
\end{equation} 
on $X=\Gamma\backslash\PSL(2,\R)$  associated to the flow $\phi_t(g)=g a_t$ on $\PSL(2,\R)$
by the conjugate relation 
\[\varphi_t^\X=\Xi^{-1}\circ\varphi_t^X\circ\Xi\quad \mbox{for all}\quad t\in\R.\]
Here $a_t\in\PSL(2,\R)$ denotes the equivalence class obtained from the matrix $A_t=\Big({\scriptsize\begin{array}{cc} 
		e^{t/2} & 0\\ 0 & e^{-t/2}
\end{array}}\Big)\in\SL(2,\R)$. 

\smallskip 
There are some more advantages to work on $X=\Gamma\backslash\PSL(2,\R)$
rather than on $\X=T^1(\Gamma\backslash\H^2)$. One can calculate explicitly the stable and unstable manifolds 
at a point $x$ to be
\begin{equation}\label{Wsu}
	W^s_X(x)=\{\theta^X_t(x),t\in\R\}
	\quad \mbox{and}\quad W^u_X(x)=\{\eta^X_t(x), t\in\R\},
\end{equation} 
where $(\theta^X_t)_{t\in\R}$ and $(\eta^X_t)_{t\in\R}$
are the {\em stable horocycle flow} and {\em unstable horocycle flow}
defined by $\theta^X_t(\Gamma g)=\Gamma g b_t$ and $\eta^X_t(\Gamma g)=\Gamma g c_t$;  
here $b_t,c_t\in\PSL(2,\R)$ denote
the equivalence classes obtained from 
$B_t=\big({\scriptsize\begin{array}{cc}1 &t\\ 0&1 \end{array} }\big), \ C_t=\big({\scriptsize\begin{array}{cc}
		1&0\\ t&1
\end{array}}\big)\in\SL(2,\R)$. 
The flow $(\varphi^X_t)_{t\in\R}$
is hyperbolic, that is, 
for every $x\in X$ there exists an orthogonal and $(\varphi_t^X)_{t\in\R}$-stable splitting of the tangent space
$T_xX$
\[T_x X= E^0(x)\oplus E^s(x)\oplus E^u(x)\]
such that the differential of the flow $(\varphi_t^X)_{t\in \R}$ is uniformly expanding on $E^u(x)$, uniformly contracting on $E^s(x)$ and isometric on $E^0(x)={\rm span}\big\{ \frac{d}{dt}\varphi_t^X(x)|_{t=0}\big\}$. One can choose  
\begin{eqnarray*}
	E^s(x)  =  {\rm span}\Big\{\frac{d}{dt}\,\theta^X_t(x)\Big|_{t=0}\Big\}
	\quad\mbox{and}\quad 
	E^u(x)  =  {\rm span}\Big\{\frac{d}{dt}\,\eta^X_t(x)\Big|_{t=0}\Big\}.
\end{eqnarray*}

General references for this section are \cite{bedkeanser,einsward,KatHas}, 
and these works may be consulted for the proofs to all results which are stated above.
In what follows, we will drop the superscript $X$ from $(\varphi^X_t)_{t\in\R}, (\theta^X_t)_{t\in\R}, (\eta^X_t)_{t\in\R}$ to simplify notation.
%In the rest of this section  we collect some useful technical results. 

\subsection{Distance on $\Gamma\backslash\PSL(2,\R)$}
\begin{lemma}\label{at}
	There is a natural Riemannian metric on $\G=\PSL(2,\R)$
	such that the induced metric function $d_\G$ is left-invariant under $\G$ and 
	\[d_\G(a_t,e)=\frac{1}{\sqrt 2}|t|, \quad d_\G(b_t,e)\leq |t|,\quad d_\G(c_t,e)\leq |t|\quad\mbox{for all}\quad t\in\R, \]
	where $e=\pi(E_2)$ is the unity of $\G$.
\end{lemma}	
See \cite[Subsection 9.3]{einsward} for more details.

We define a metric function $d_{X}$ on $X=\Gamma\backslash\PSL(2,\R)$ by 
\begin{equation}\label{dX} d_{X}(x_1, x_2)
	=\inf_{\gamma_1, \gamma_2\in\Gamma} d_{\G}(\gamma_1 g_1, \gamma_2 g_2)
	=\inf_{\gamma\in\Gamma} d_{\G}(g_1, \gamma g_2), 
\end{equation}
where $x_1=\Gamma g_1 $, $x_2=\Gamma g_2$.
In fact, if $X$ is compact, one can prove that the infimum is a minimum:
\begin{equation*}
	d_{X}(x_1, x_2)=\min_{\gamma\in\Gamma} d_{\G}(g_1, \gamma g_2).
\end{equation*} 

\begin{lemma}\label{tslm}
	For any $x\in X$ and $t,s\in \R$, one has
	\[d_X(\varphi_{t}(x),\varphi_{s}(x))\leq \frac{1}{\sqrt 2}|t-s|. \]
\end{lemma}
\begin{proof}
	Suppose $x=\Gamma g$ for some $g\in \PSL(2,\R)$, then 
	\begin{align*}
		d_X(\varphi_{t}(x),\varphi_{s}(x))&=
		d_X(\Gamma g a_{t},\Gamma g a_{s})
		\leq d_\G( g a_{t}, g a_{s})\\
		&=
		d_\G(a_{t},a_s)=d_\G(a_{t-s},e)=\frac{1}{\sqrt 2}|t-s|.
	\end{align*}
\end{proof}

It is well-known that the Riemann surface $\Gamma\backslash\H^2$ is compact if and only if 
the quotient space  $X=\Gamma\backslash\PSL(2,\R)$ is compact. It is possible to derive 
a uniform lower bound on $d_{\G}(g, \gamma g)$ 
for $g\in \PSL(2,\R)$ and $\gamma\in\Gamma\setminus\{e\}$.

\begin{lemma}\label{sigma_0}
	If the space $X=\Gamma\backslash\PSL(2,\R)$ is compact, then there exists $0<\sigma_*<1$ such that
	\[d_\G(\gamma g, g)>\sigma_*\quad\mbox{for all}\quad \gamma\in\Gamma\setminus\{e\}. \]
\end{lemma}
The number $\sigma_*$ is called an injectivity radius. See \cite[Lemma 1, p. 237]{rat} for a similar result on $\Gamma\backslash\H^2$.

\subsection{Poincar\'e sections}\label{Poincsec} 

\begin{definition}[Poincar\'e section]\label{Poindn1}\rm 
	Let $x\in X$ and $\eps>0$. The {\em (closed)  Poincar\'e section} of radius $\eps$ at $x$ is defined by
	\begin{equation*}
		P_\eps(x)=\{(\theta_s\circ \eta_u)(x): |u|\leq\eps, |s| \leq\eps\}=\{\Gamma g c_ub_s : |u|\leq\eps, |s| \leq\eps\},
	\end{equation*}
	where $g\in \G$ is such that $x=\Gamma g$; see Figure \ref{lcs2}\,(a) for an illustration.
	
	Another version of Poincar\'e section is 	
	\[\widetilde P_\eps(x)=\{(\eta_u\circ \theta_s)(x): |u|\leq\eps, |s| \leq\eps\}=\{\Gamma g b_s c_u: |s|\leq \eps,|u|\leq \eps \}. \] 
\end{definition}
Note that both sets do not depend on the choice of $g\in \G$ such that $x=\Gamma g$. Similarly to open Poincar\'e sections in \cite{partII} one can coordinalize Poincar\'e sections in the case that the radius is small enough.
\begin{lemma} Let $X$ be compact, $\eps\in\, (0,\sigma_*/4)$, and $x=\Gamma g$ for $g\in \G$. 
	
	\noindent
	(a) For every $y\in P_\eps(x)$ there exists a unique couple  $(u,s)\in\, [-\eps,\eps]\,\times\,[-\eps,\eps]$  such that
	$y=\Gamma gc_u b_s$. 
	Then we write $y=(u,s)_x$.
	
	\noindent 	(b) For every $y\in \widetilde P_\eps(x)$ there exists a unique couple  $(s,u)\in\, [-\eps,\eps]\,\times\,[-\eps,\eps]$  such that $y=\Gamma g b_sc_u$. 
	Then we write $y=(s,u)'_x$.
\end{lemma}
See \cite[Lemma 2.1]{partII} for a proof.

\subsection{Some auxiliary results}
\begin{lemma}\label{decompo}
	Let $g=[G]\in {\rm PSL}(2, \R)$ for $G=\scriptsize \Big(\begin{array}{cc} a & b \\ c & d\end{array}\Big)
	\in {\rm SL}(2, \R)$. 
	\begin{itemize}
		\item[(a)] If $a\neq 0$, then $g=c_u b_s a_t$ for 
		\begin{equation}\label{ust-def}
			t=2\ln |a|,\quad s=ab,\quad u=\frac{c}{a}. 
		\end{equation} 
		\item[(b)] If $d\neq 0$, then $g=a_tb_s c_u$ for 
		\[ t=-2\ln |d|,\quad s={d},\quad u=\frac{c}d. \] 
	\end{itemize} 
\end{lemma} See \cite[Lemma 2.3]{HK} for a proof of (a). A similar argument can be applied for (b). 
\begin{lemma}\label{konvexa}
	For every $\eps>0$ there is $\delta=\delta(\eps)>0$ with the following property. 
	If $g\in\PSL(2,\R)$ is such that $d_{\G}(g,e)<\delta$, then there are
	\[ G=\Bigg(\begin{array}{cc}g_{11}&g_{12}\\
		g_{21}&g_{22}
	\end{array}\Bigg)\in\SL(2,\R) \]
	satisfying $g=\pi(G)$ and
	$|g_{11}-1|+|g_{12}|+|g_{21}|+|g_{22}-1|<\eps$.
\end{lemma}	
See \cite[Lemma 2.17]{HK} for a proof.
\begin{lemma}\label{eps0}
	For every $\eps>0$, there exists  $\rho=\rho(\eps)>0$ with the following property. 
	For $x=\Gamma g$, $z=\Gamma g c_ub_s$ with $|s|,|u|<\sigma_*/8$ and $L>0$. Then
	\begin{enumerate}
		\item[(a)] if  $d_X(\varphi_t(x),\varphi_t(z))<3\rho$ for $t\in [-L,0]$, then $|s|<\eps e^{-T}$;
		\item[(b)]  if $d_X(\varphi_t(x),\varphi_t(z))<3\rho$ for $t\in [0,L]$, then $|u|<\eps e^{-T}$.
	\end{enumerate}	
\end{lemma}
See \cite[Theorem 2.1]{partII} for a proof.

%%%%%%%%%%%%%%%%%%%%%%%%%%%%%%%%%%%%%%%%%%%%%%%%%%%%%%%

\section{Local product structure} 
\label{lpss}

In this section we construct local product structure for the system.
We use explicit forms of (local) stable and unstable manifolds 
to calculate local product structure in detail.

%\begin{definition}%[Weak-stable and weak-unstable manifold]\label{westunma} \rm 
%	Let $x\in X$. 
%	The {\rm weak-stable manifold} and {\rm weak-unstable manifold} at $x$ are given by 
%	\[ W^{ws}(x)=\{(\theta_s\circ\varphi_t)(x): s, t\in\R\}
%	=\{\Gamma ga_t b_s: s, t\in\R\} \] 
%	and 
%	\[ W^{wu}(x)=\{(\eta_u\circ\varphi_t)(x): u, t\in\R\}
%	=\{\Gamma ga_t c_u: u, t\in\R\}, \] 
%	respectively, where $g\in \G$ is such that $x=\Gamma g$. 
%\end{definition}
\begin{definition}\rm 
	Let $x\in X$ and $\eps>0$. The local stable and local unstable manifold at $x$ are given by 
	\[W_\eps^s(x)=\{ \theta_s(x): |s|<\eps\}=\{ \Gamma g b_s: |s|<\eps \}\]
	and 
	\[W_\eps^u(x)=\{\eta_u(x):|u|<\eps \}=\{ \Gamma gc_u: |u|<\eps\}. \]
\end{definition}
Note that both sets are independent of the choice of $g\in \PSL(2,\R)$ such that $x=\Gamma g$. 
We also need the notion of local weak-stable and local weak-unstable manifold:
\[W^{ws}_\eps(x)=\{ \Gamma ga_tb_s: |t|<\eps,|s|<\eps \}, \]
\[W^{wu}_\eps(x)=\{\Gamma g a_tc_u: |t|<\eps, |u|<\eps \}. \]

\begin{lemma}\label{prodstr} 
	Let $\eps\in (0, \sigma_*/5)$. There exists a $\delta=\delta(\eps)>0$ with the following property. 
	If $x, y\in X$ satisfy $d_X(x, y)<\delta$, then the intersection
	\[ W^{ws}_{\eps}(x)\cap W^u_{\eps}(y) \] 
	consists of a unique point, and furthermore the intersection
	\[ W^{wu}_{\eps}(x)\cap W^s_{\eps}(y) \] 
	consists of a unique point. 
\end{lemma} 
\begin{proof} We prove the first assertion only. 
	In order to show that such a $z\in W^{ws}_{\eps}(x)\cap W^u_{\eps}(y)$
	does exist, note that according to Lemma \ref{konvexa} there is $\delta=\delta(\eps)>0$ so that the following holds. 
	If $u\in \G$ and $d_\G(u, e)<\delta$, then there is 
	\begin{equation}\label{parto} 
		A=\bigg(\begin{array}{cc} a & b
			\\ c & d\end{array}\bigg)\in {\rm SL}(2, \R)
	\end{equation} 
	such that $u=\pi(A)$ and $|a-1|+|b|+|c|+|d-1|<\min\{\frac{1}{2}, \frac{\eps}{4}\}$. 
	Fix $x, y\in X$ with $d_X(x, y)<\delta$. Let $g\in \G$ and $h\in \G$ satisfy $x=\Gamma g$ 
	and $y=\Gamma h$ as well as $d_X(x, y)=d_\G(g, h)$. Then 
	\[ d_\G(g^{-1}h, e)=d_\G(g, h)=d_X(x, y)<\delta, \] 
	and hence there is $A\in {\rm SL}(2, \R)$ as in (\ref{parto}) 
	such that $g^{-1}h=[A]$ and $|a-1|+|b|+|c|+|d-1|<\min\{\frac{1}{2}, \frac{\eps}{4}\}$; 
	then in particular $d\in [1/2, 3/2]$ holds. We can write $g^{-1}h=a_t b_s c_u$ for 
	\[ t=-2\ln d,\quad s=bd,\quad u=\frac{c}{d}. \]
	Then $hc_{-u}=g a_t b_s$ and also $|t|=2|\ln d|\le 4|d-1|<\eps$ due to $|\ln(1+z)|\le 2|z|$ 
	for $|z|\le 1/2$. Furthermore, $|s|=|b||d|\le 2|b|<\eps/2$ 
	and $|u|=\frac{|c|}{|d|}\le 2|c|<\eps/2$. Therefore if we put 
	$z=\Gamma g a_t b_s=\Gamma hc_{-u}\in X$, 
	then $z\in W^{ws}_{\eps}(x)\cap W^u_{\eps}(y)$. 
	It remains to prove that the intersection point is unique. 
	To establish this assertion, suppose that also 
	$z'\in W^{ws}_{\eps}(x)\cap W^u_{\eps}(y)$. 
	Then $z'=\Gamma ga_{t'}b_{s'}=\Gamma hc_{-u'}$ 
	for some $|t'|, |s'|, |u'|<\eps$. Hence $\Gamma hc_u a_t b_s
	=\Gamma g=\Gamma hc_{u'}a_{t'}b_{s'}$, which means that 
	$hc_u a_t b_s=\gamma hc_{u'}a_{t'}b_{s'}$ for an appropriate element 
	$\gamma\in\Gamma$. This yields
	\begin{eqnarray*}
		d_\G(\gamma hc_{u'}a_{t'}b_{s'}, hc_{u'}a_{t'}b_{s'})
		& = & d_\G(hc_u a_t b_s, hc_{u'}a_{t'}b_{s'})
		=d_\G(c_u a_t b_s, c_{u'}a_{t'}b_{s'})
		\\ & \le & d_\G(c_u a_t b_s, e)+d_\G(c_{u'}a_{t'}b_{s'}, e)
		\\ & \le & |u|+\frac{1}{\sqrt{2}}|t|+|s|+|u'|+\frac{1}{\sqrt{2}}|t'|+|s'|
		<5\eps<\sigma_*.
	\end{eqnarray*} 
	From the property of $\sigma_*$ we deduce that $\gamma=e$ 
	and therefore $c_u a_t b_s=c_{u'}a_{t'}b_{s'}$. 
	Multiplying out the matrices  
	we obtain $u=u'$, $t=t'$, and $s=s'$, and accordingly 
	$z=\Gamma gb_{-s}=\Gamma gb_{-s'}=z'$. \end{proof}

\begin{corollary}[Local product structure]\label{lpslm}
	Let  $\eps\in (0, \sigma_*/5)$. There exists a positive number $\delta=\delta(\eps)$ with the following property. 
	If $x,y\in X$ and $d_X(x,y)\leq \delta$, then 
	there is a unique $v=v(x,y)\in \R$, $|v|\leq \eps$ such that 
	\[W^s_\eps(\varphi_v(x))\cap W^u_\eps(y) \ne \varnothing.\]
	More precisely, the intersection is a single point, denoted by $\langle x,y\rangle$. 
	Furthermore, the map $\langle \cdot, \cdot\rangle$ is continuous on 
	$\{ (x,y)\in X\times X: d_X(x,y)<\delta\}. $
\end{corollary}
\begin{proof} Let $\eps\in (0,\sigma_*/5)$ and let $\delta=\delta(\eps)$ be as in Lemma \ref{prodstr}.
	Let $x, y\in X$ be such that $d_X(x, y)<\delta$. 
	Then $W^{ws}_{\eps}(x)\cap W^u_{\eps}(y)\neq\varnothing$, 
	i.e., if $x=\Gamma g$ and $y=\Gamma h$, then there are $s,v,u\in (-\eps,\eps)$ 
	such that $\Gamma ga_v b_s=\Gamma h c_u $. Since $\varphi_v(x)=\Gamma g a_v$ 
	we have 
	\[ W^s_{\eps}(\varphi_v(x))=\{\Gamma ga_v b_{s'}: |s'|<\eps\}, \] 
	and hence $\Gamma ga_t b_s=\Gamma h c_u\in W^s_{\eps}(\varphi_v(x))
	\cap W^u_{\eps}(y)$. If also $W^s_{\eps}(\varphi_{v'}(x))\cap W^u_{\eps}(y)
	\neq\varnothing$ for some $|v'|<\eps$, then $\Gamma ga_{v'} b_{s'}=\Gamma h c_{u'}$ 
	for appropriate $|s'|, |u'|<\eps$. From Lemma \ref{prodstr} we obtain 
	$v=v'$, $s=s'$, and $u=u'$, so that $v$ is unique. That the intersection 
	is a single point also follows from Lemma \ref{prodstr}; see Figure \ref{lps} for an illustration.
	The last assertion is obvious. 
\end{proof}\bigskip
\begin{figure}[ht]
	\begin{center}
		\begin{minipage}{0.6\linewidth}
			\centering
			\includegraphics[angle=0,width=0.7\linewidth]{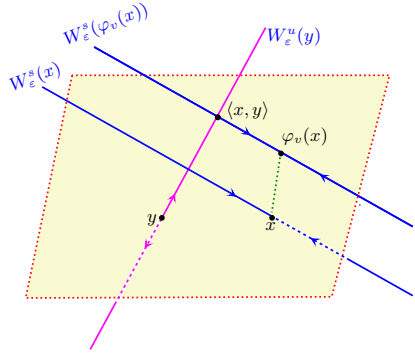}
		\end{minipage}
	\end{center}
	\caption{Local product structure}\label{lps}
\end{figure}

Fix $\eps\in (0,\sigma_*/5)$ and let $\delta_1=\delta(\eps)$ from
Corollary \ref{lpslm} above. Define $\delta_2=\min \{\delta(\frac{\delta_1}3),\frac{\delta_1}3\}$.
We also have a similar result to \cite[Lemma 6]{bo-diff}.
\begin{lemma}\label{xyz}
	Let $x,y,z,w\in X$ be such that $\diam \{x,y,z,w \}<\delta_2$. Then
	
	(a) $\langle\langle x, y\rangle,z\rangle=\langle x,\langle y,z\rangle\rangle=\langle x, z\rangle$;
	
	(b) $\langle\langle x,y\rangle,\langle z, w\rangle\rangle=\langle x,w\rangle$.
\end{lemma}
\begin{proof}  
	For $x,y,z,w\in X$, we first check that all the notations make sense if $\diam\{x,y,z,w\}<\delta_2$.
	Obviously $\langle a,b\rangle$ makes sense for all $a,b\in\{x,y,z,w\}$. 
	Write $y=\Gamma g$ and $w=\Gamma h$ for $g,h\in\PSL(2,\R)$.
	By the proof of Lemma \ref{prodstr}, 
	$\langle x,y\rangle=\Gamma gc_{u_1}$ for some $|u_1|<\delta_1/3$
	and $\langle z,w\rangle=\Gamma hc_{u_2}$ for some $|u_2|<\delta_1/3$.
	Then 
	\begin{align*}
		d_X(\langle x,y\rangle,z)
		&=
		d_X(\Gamma gc_{u_1}, z)\leq d_X(\Gamma g c_{u_1}, \Gamma g)+d_X(y,z)
		\\&
		\leq |u_1|+d_X(y,z)
		<\delta_1/3+\delta_1/3<\delta_1
	\end{align*}
	so
	$\langle\langle x, y\rangle,z\rangle$ makes sense.
	Similarly, $\langle x,\langle y,z\rangle\rangle$ also makes sense. 
	Next, 
	\begin{align*}
		d_X(\langle x,y\rangle,\langle z, w\rangle)
		&=d_X(\Gamma g c_{u_1},\Gamma h c_{u_2})
		\leq d_X(\Gamma g c_{u_1}, \Gamma g)+d_X(y,w)+d_X(\Gamma h, \Gamma h c_{u_2})
		\\
		&
		<|u_1|+\delta_1/3+|u_2|<\delta_1
	\end{align*} 
	and hence
	$\langle\langle x,y\rangle,\langle z, w\rangle\rangle$ also makes sense. 
	
	(a) Note that if $x'\in W^{ws}_{\frac{\delta_1}3}(x)$, then $W^{ws}_{\frac{\delta_1}3}(x')\subset W^{ws}_{\delta_1}(x)$
	and if $z'\in W^u_{\frac{\delta_1}3}(z)$, then $W^u_{\frac{\delta_1}3}(z')\subset W^u_{\frac{2\delta_1}{3}}(z)$.  
	By Lemma \ref{prodstr}, $\langle x,y\rangle \in W^{ws}_{\frac{\delta_1}3}(x)$ and 
	\[\langle\langle x, y\rangle,z\rangle \in W^{ws}_{\frac{\delta_1}3}(\langle x,y\rangle)\cap W^u_{\frac{\delta_1}3}(z)
	\in W^{ws}_{\delta_1}(x)\cap W^u_{\delta_1}(z)=\langle x,z\rangle.\]
	Similarly, 
	$\langle y,z\rangle \in W^{u}_\eps(z)$ implies 
	\[\langle x,\langle y,z\rangle\rangle\in W^{ws}_{\frac{\delta_1}3}(x)\cap W^u_{\frac{\delta_1}3}(\langle y,z\rangle)
	\in W^{ws}_{\frac{\delta_1}3}(x)\cap W^u_{\frac{2\delta_1}3}(z)=\langle x,z\rangle. \]
	(b) Applying (a), we have $\langle\langle x,y\rangle,\langle z, w\rangle\rangle
	=\langle \langle x,y\rangle, w\rangle 
	=\langle x,w\rangle$.
\end{proof}

\section{Local cross sections and rectangles}\label{lcssection}

This section deals with rectangles included in  
Poincar\'e sections. We introduce explicit forms of rectangles 
that leads to more explicit Markov partition afterwards. 

\subsection{Local cross sections} 
\begin{definition}[Local cross section]\rm A set
	$S\subset X$ is  called a {\em cross section of time} $\eps>0$ for the flow $(\varphi_t)_{t\in\R}$  if
	\begin{enumerate}
		\item[(a)] $S$ is closed;
		\item[(b)] $S\cap \varphi_{[-\eps,\eps]}(x)=\{x\}$ for all $x\in S$.
	\end{enumerate}
\end{definition}
See Figure \ref{lcs2}\,(a) for an illustration.

\begin{figure}[ht]
	\begin{center}
		\begin{minipage}{0.6\linewidth}
			\centering
			\includegraphics[angle=0,width=1\linewidth]{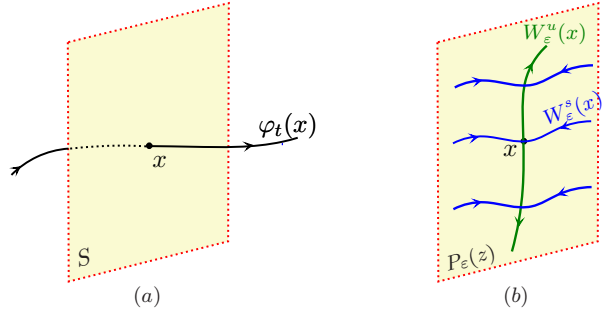}
		\end{minipage}
	\end{center}
	\caption{(a)  Local cross section,\  (b) Poincar\'e section\label{lcs2}}
\end{figure}

We consider an example of local cross sections.

\begin{lemma}\label{lcexpl} Let $\eps>0, \alpha>0$ be such that 
	$4\eps+2\alpha<\sigma_*$ and let $z=\Gamma g\in X$. The closed Poincar\'e sections 
	\begin{align*}
		P_{\eps}(z)&=\{\Gamma gc_ub_s,\ |u|\leq \eps, |s|\leq \eps\}
		\\
		\widetilde{P}_{\eps}(z)&=\{\Gamma gb_sc_u,\ |s|\leq \eps, |u|\leq \eps\}
	\end{align*}
	are  local cross sections of time $\alpha$ and with diameters at most $ 4\eps$. 
\end{lemma}
%\begin{lemma}\label{lcexpl} 
%	For $\eps\in(0,\sigma_*/6)$ and $z=\Gamma g\in X$,	the closed Poincar\'e sections 
%	\begin{align*}
%		P_{\eps}(z)&=\{\Gamma gc_ub_s,\ |u|\leq \eps, |s|\leq \eps\}
%		\\
%		\widetilde{P}_{\eps}(z)&=\{\Gamma gb_sc_u,\ |s|\leq \eps, |u|\leq \eps\}
%	\end{align*}
%	are  local cross sections of time $\eps$ and with radii not greater than $ 4\eps$. 
%\end{lemma}
\begin{proof}  Obviously, $P_\eps(z)$ is closed. 
	Note that
	\[{\mathcal Q}:=\varphi_{[-\eps,\eps]}(P_\eps(z))=\{\Gamma gc_ub_sa_t, |u|\leq \eps, |s|\leq\eps, |t|\leq \alpha\}.\]
	In order to verify Assumption (b), we check that every point $x\in {\mathcal Q}$ has a unique triple $(u,s,t)\in[-\eps,\eps]^2\times [-\alpha,\alpha]$ such that
	$x=\Gamma g c_ub_sa_t$.
	To show its uniqueness, suppose that $z=\Gamma g_1=\Gamma g_2$ and $x=\Gamma g_1b_{s_1}c_{u_1}a_{t_1}=\Gamma g_2 b_{s_2}c_{u_2}a_{t_2}$ for 
	$g_1,g_2\in \G$ and $(u_i,s_i,t_i)\in[-\eps,\eps]^2\times [-\alpha,\alpha]$. Then there are $\gamma, \gamma' \in\Gamma$ such that
	\[\gamma g_1= g_2\quad\mbox{and}\quad \gamma' g_1c_{u_1}b_{s_1}a_{t_1}=g_2c_{u_2}b_{s_2}a_{t_2}.\]
	Therefore,
	\begin{eqnarray*}
		\lefteqn{ d_{\G}(\gamma^{-1}\gamma' g_1 c_{u_1}b_{s_1}, g_1 c_{u_1}b_{s_1})}
		\\
		&=&d_{\G}(\gamma^{-1}g_2c_{u_2}b_{s_2}a_{t_2-t_1},g_1c_1b_{s_1})
		=d_{\G}( g_1 c_{u_2}b_{s_2}a_{t_2-t_1}, g_1c_{u_1}b_{s_1})\\
		&=&d_{\G}(c_{u_2-u_1} b_{s_2}a_{t_2-t_1}, b_{s_1})
		\leq d_{\G}(c_{u_2-u_1} b_{s_2}a_{t_2-t_1},e)+d_{\G}(b_{s_1},e)\\
		&\leq& |u_2-u_1|+|s_2|+|t_2-t_1|+|s_1|<4\eps+2\alpha<\sigma_*.
	\end{eqnarray*}
	From the property of $\sigma_*$, this implies that $\gamma^{-1}\gamma'=e$, so that  $\gamma=\gamma'$. 
	Then $g_2c_{u_2}b_{s_2}a_{t_2}=\gamma g_1c_{u_1}b_{s_1}a_{t_1}=g_2c_{u_1}b_{s_1}a_{t_1}$
	yields $c_{u_1}b_{s_1}a_{t_1}=c_{u_2}b_{s_2}a_{t_2}$,
	and consequently $u_2=u_1,s_2=s_1,t_2=t_1$ by considering matrices.
	This leads to $\varphi_{[-\alpha,\alpha]}(x)\cap P_\eps(z)=\{x\}$ and hence $P_\eps(z)$ is a local cross section of time $\eps$.
	For the last assertion, if $x=\Gamma gc_{u_x}b_{s_x},y=\Gamma gc_{u_y}b_{s_y}\in P_\eps(z)$, then
	\[d_X(x,y)\leq d_\G(c_{u_x}b_{s_x},c_{u_y}b_{s_y})\leq |u_x|+|s_x|+|u_y|+|s_y|\leq 4\eps\]
	shows $\diam P_\eps(z)\leq 4\eps$. The same argument can be applied for $\widetilde P_\eps(z)$.
\end{proof}\bigskip
By the same manner as in the previous proof, it follows the next result.
\begin{proposition}
	
	(a) Let ${\tt u}>0, {\tt s}>0$ and $\alpha>0$ be such that $2{\tt u}+2{\tt s}+\alpha<\sigma_*$ and let $z=\Gamma g\in X$. The sets
	\begin{align*} 
		P_{\tt s}^{\tt u}(z)=\{\Gamma gc_ub_s,\ |u|\leq {\tt u}, |s|\leq {\tt s}\},  \ \widetilde{P}_{\tt u}^{\tt s}(z)=\{\Gamma gb_sc_u,\ |s|\leq {\tt s}, |u|\leq {\tt u}\}
	\end{align*}
	are local cross sections of time $\alpha$ and with diameters at most $2({\tt u}+{\tt s})$.
	
	(b) Let $\eps>0,\alpha>0, \tau\in \R$ be such that $2\eps (e^{\tau}+e^{-\tau})+\alpha<\sigma_*$ and let $z=\Gamma g\in X$. The sets
	\[\varphi_{\tau}(P_\eps(z))=\{\Gamma ga_\tau c_{ue^\tau}b_{se^{-\tau}}: |u|\leq \eps,|s|\leq \eps \}
	=P_{\eps e^{-\tau}}^{\eps e^\tau}(\varphi_\tau(z)),\]
	\[ \varphi_{\tau}(\widetilde{P}_{\eps}(z)) =\{ \Gamma g a_\tau b_{se^{-\tau}}c_{u e^\tau}: |s|\leq \eps, |u|\leq \eps\}=\widetilde{P}_{\eps e^{-\tau}}^{\eps e^\tau}(\varphi_\tau(z))\]  are local cross sections of time 
	$\alpha$ and with diameters at most $2\eps(e^{\tau}+e^{-\tau})$. 
\end{proposition}
For a general flow, the following result is not obvious, see \cite{wn}. However, for the geodesic flow on compact factors of the hyperbolic plane, it is quite simple.

\begin{proposition}\label{lmnu}
	For $x\in X$, there is a local cross section $S_x$ of time $\nu_x>0$ so that $x\in \int S_x$.
\end{proposition}
\begin{proof} We can choose $S_x=P_{\nu_x}(x)=\{\Gamma g c_ub_s: |s|,|u|\leq\nu_x \}$ for $\nu_x\in (0,\sigma_*/6)$;
	see Lemma \ref{lcexpl}.
\end{proof}

It is clear that if $S$ is a local cross section of time $\eps$, then $S\times [-\alpha,\alpha]$ is homeomorphic with the compact set $\varphi_{[-\alpha,\alpha]}(S)$.

\begin{definition}[Projection map]\label{prodn}\rm  Let $S$ be a local cross section of time $\alpha$. 
	The map \[\pr_S: \varphi_{[-\alpha,\alpha]}(S)\to S,\quad 
	\pr_S(\varphi_t(x))=x\quad\mbox{for all}\quad t\in[-\alpha,\alpha]\] is called the {\em projection map} to $S$.
\end{definition}

Let $\eps\in (0,\sigma_*/6)$ and $\delta=\delta(\eps)$ be as in Corollary \ref{lpslm}. If  $D$ is a local cross section of time $\eps$ and $T\subset D$ is a closed set such that
$\diam T<\delta$ and $\diam T$ and $d(T, \partial D)>0$. We assume that $\pr_D(\langle x,y\rangle)$ do exist for all $x,y\in T$.
We define 
\begin{equation}\label{projeq} 
	\langle\cdot ,\cdot \rangle_D: T\times T\longrightarrow D, \quad \langle x, y \rangle_D=\pr_D(\langle x, y\rangle).
\end{equation}
See Figure \ref{rec} for an illustration. It is worth mentioning that $\langle x,y\rangle_D\in D$ and may not be in $T$ and $\langle \cdot,\cdot\rangle_D: T\times T\to D$ is continuous.

From now on, we fix $\eps\in (0,\sigma_*/5)$ and  $\delta=\delta(\eps)$ be from Corollary \ref{lpslm}.  The next result determines $\langle x,y\rangle_D$ precisely.

\begin{lemma}\label{proexa}
	Let $D=P_{2\alpha}(z)$ and $T=P_{\alpha/4}(z)$ for $\alpha\in (0,\delta)$ and  $z=\Gamma g\in X$. 
	If $x=\Gamma g c_{u_x}b_{s_x},y=\Gamma g c_{u_y}b_{s_y}\in T$,
	then $$\langle x,y\rangle=\Gamma g c_{u_x}b_{s_x+s}a_v=\Gamma h c_{u_y}b_{s_y}c_{u}, \mbox{\ \ } 
	\langle x,y\rangle_D=\Gamma g c_{u_x}b_{s_w}=\Gamma gc_{u_y}b_{s_y} c_{u}a_{-v}\in D,$$ 
	where $s,u,v, s_w$ are defined by
	\begin{equation}\label{proexaeq} 
		\begin{aligned}
			&s=(s_y-s_x-s_xs_y(u_y-u_x))(1+(u_y-u_x)s_y), \quad 
			u=\frac{u_x-u_y}{1+(u_y-u_x)s_y},
			\\
			& v= -2\ln(1+(u_y-u_x)s_y),\quad 
			s_w=\frac{s_y}{1+(u_y-u_x)s_y}.
		\end{aligned}
	\end{equation}
	
\end{lemma} 
\begin{proof} Since $\diam T\leq \alpha<\delta$, $\langle x,y\rangle$ do exist for all $x,y\in T$.  First we have
	\[B_{-s_x}C_{u_y-u_x}B_{s_y}= \begin{pmatrix}
		1- (u_y-x_x)s_y& s_x-s_x-s_xs_y(u_y-u_x)\\
		u_y-u_x&1+(u_y-u_x)s_y
	\end{pmatrix}\] 
	and $\pi(B_{-s_x}C_{u_y-u_x}B_{s_y})=b_{-s_x}c_{u_y-u_x}b_{s_y}=(c_{u_x}b_{s_x})^{-1}c_{u_y}b_{s_y}$. 
	Let $s,u,v,s_w$ be defined by \eqref{proexaeq}. Then $s,u,v,s_w\in [-\eps,\eps]$. By Lemma \ref{decompo}\,(b), 
	$b_{-s_x}c_{u_y-u_x}b_{s_y}=a_vb_sc_{-u}$. This implies $\Gamma gc_{u_x}b_{s_x}a_vb_s=\Gamma gc_{u_y}b_{s_y}c_u\in W^s_{\eps }(\varphi_v(x))\cap W^u_\eps(y)
	=\langle x,y\rangle$.
	Furthermore, $\langle x,y\rangle=\Gamma g c_{u_x}b_{s_x}s_{se^v}a_v=\Gamma gc_{u_x}b_{s_w}a_v$ after a short calculation.
	Consequently,  $\varphi_{-v}(\langle x,y\rangle)=\Gamma g c_{u_y}b_{s_y}c_ua_{-v}=\Gamma g c_{u_x}b_{s_w}\in D$ yields 
	$\langle x,y\rangle_D=\Gamma g c_{u_x}b_{s_w}=\Gamma g c_{u_y}b_{s_y}c_ua_{-v}$, proving the lemma. 
\end{proof}

\subsection{Rectangles}
In this subsection, we fix $\eps_*\in (0,\sigma_*/5)$ and $\delta_*=\delta(\eps_*)$ from Corollary \ref{lpslm}.

It was mentioned in the last subsection that for a local cross section $D$ and $R\subset D$, if
$x,y\in R$  then
$\langle x,y\rangle_D$ may not be in $R$.

\begin{definition}[Rectangle]\rm Let $D$ be a local cross section  and 
	$\diam D<\delta$. A subset $\varnothing \ne R\subset D$ is called a {\em rectangle} if
	\begin{itemize}
		\item [{($R_1$)}] $R$ is closed in $D$;
		\item[{($R_2$)}] $\langle x, y \rangle_D\in R$ for all $x,y\in R$. 
	\end{itemize}
\end{definition}
See Figure \ref{rec} for an illustration. In the case that $R$ is a rectangle, for $x,y\in R$ we can write $\langle x,y\rangle_R$
for $\langle x,y\rangle_D$ since it does not depend on $D$.

\begin{figure}[ht]
	\begin{center}
		\begin{minipage}{0.6\linewidth}
			\centering
			\includegraphics[angle=0,width=0.6\linewidth]{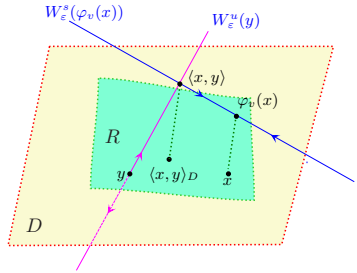}
		\end{minipage}
	\end{center}
	\caption{Rectangle $R$ in local cross section $D$}\label{rec}
\end{figure}
\begin{remark}\label{recrm}\rm (a)
	If $R_1\subset D_1$ and $R_2\subset D_2$ are rectangles
	and $R_1\cap R_2\ne \varnothing$, then 
	$R_1\cap R_2$ is a rectangle.

	\medskip 	\noindent
	(b) If $R\subset D$ is a rectangle, then
	so is $\varphi_\tau(R)$ for appropriately small $\tau\in \R$. Indeed, assume that $x,y\in \varphi_\tau(R)$. Then
	$x=\varphi_\tau(x')$ and $y=\varphi_\tau(y')$ for $x', y'\in R$.
	Write $x'=\Gamma g$ and $y'=\Gamma h$. Assume that
	$\langle x', y'\rangle =
	\Gamma g a_vb_s=\Gamma h c_u$ for some small numbers $v,s,u\in\R$. Then 
	$\langle x, y\rangle
	=\Gamma ga_\tau a_{v'}b_{s'}
	=\Gamma h a_\tau c_{u'}$
	with $v'=v, s'=se^{-\tau}$ and $u'=u e^\tau$.
	This yields
	$\langle x,y\rangle=\Gamma g a_vb_s a_\tau =\varphi_\tau(\langle x',y'\rangle)$.
	If $\varphi_{\lambda}(\langle x',y'\rangle)= \langle x',y'\rangle_{R}\in \R$, then $\varphi_{\lambda}(\langle x,y\rangle)
	=\varphi_{\tau}(\varphi_\lambda(\langle x',y'\rangle))\in  \varphi_\tau(R)$
	implies that $\varphi_{\lambda}(\langle x,y\rangle)=\langle x,y\rangle_{\varphi_\tau(R)}\in \varphi_\tau(R)$.
	{\hfill$\Diamond$}
\end{remark}
The next result gives us an explicit example of rectangles; see  Figure \ref{rect}
for an illustration.

\begin{figure}[ht]
	\begin{center}
		\begin{minipage}{0.6\linewidth}
			\centering
			\includegraphics[angle=0,width=0.7\linewidth]{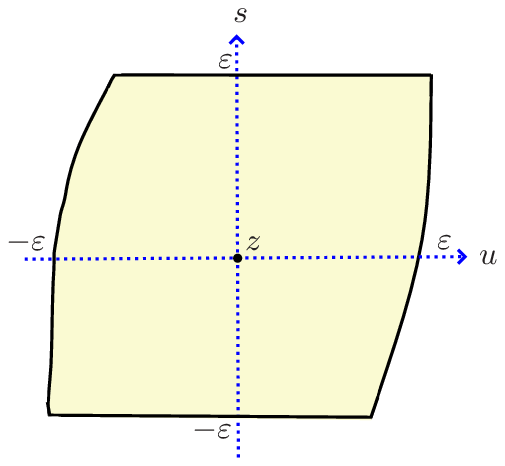}
		\end{minipage}
	\end{center}
	\caption{Rectangle $S_{\eps}(z)$}\label{rect}
\end{figure}

\begin{proposition}[Rectangle]\label{recRthm} Let $\eps\in(0,1)$ be such that $\frac\eps{1-\eps^2}\in (0,\delta_*/4)$ and $z\in X$. The sets
	\begin{eqnarray*}
		S_{\eps}(z)
		&=&
		\Big\{\Gamma g c_u b_s\,:\, u\in [-\eps,\eps],\, s=\frac{ s'}{1-us'}\quad \mbox{for some}\quad s'\in [-\eps,\eps] \Big\},\\
		T_{\eps}(z)
		&=&
		\Big\{\Gamma g b_s c_u\,:\, s\in [-\eps,\eps],\, u=\frac{ u'}{1-su'}\quad \mbox{for some}\quad u'\in [-\eps,\eps] \Big\}
	\end{eqnarray*}
	are rectangles, where $g\in\PSL(2,\R)$ is such that $z=\Gamma g$.
\end{proposition}

\noindent
\begin{proof} We only prove for $ S:=S_{\eps}(z)\subset P_{2\eps}(z)=:P$. Note that from the assumption, $\diam S<\delta_2$.
	Take $x=\Gamma g c_{u_x}b_{s_x},\ y=\Gamma gc_{u_y} b_{s_y}\in S$.
	Then 
	$\langle x, y\rangle 
	=W^s_\eps(\varphi_v(x))\cap W^u_\eps(y)=\Gamma g c_{u_x}b_{s_x}a_vb_s=\Gamma g c_{u_y}b_{s_y}c_u$
	with 
	\begin{align*}
		v=-2\ln(1+(u_y-u_x)s_y),\
		s&=(s_y-s_x-(u_y-u_x)s_xs_y)(1+(u_y-u_x)s_y),\\
		u&=\frac{u_y-u_x}{1+(u_y-u_x)s_y}.
	\end{align*}
	Rewriting 
	$\langle x,y\rangle=\Gamma g c_{u_x}b_{s_x}a_vb_s=\Gamma g c_{u_x}b_{s_x+se^v}a_v$
	implies that $\langle x,y\rangle_{P}=\Gamma g c_{u_x}b_{s_x+se^v}$. 
	We need to verify $\Gamma g c_{u_x}b_{s_x+se^v}\in S$. 
	A short calculation shows that 
	$s_x+se^v=\frac{s_y}{1+(u_y-u_x)s_y}=\frac{s_y'}{1-u_x s_y'}$, where $s_y'\in [-\eps,\eps]$
	satisfies $s_y=\frac{s_y'}{1-u_ys_y'}$; and thus
	$\langle x, y\rangle_{P} \in S$, which completes the proof. 
	%
	%\noindent
	%(b)
	%Take $x=\Gamma gb_{s_x}c_{u_x},y=\Gamma gb_{y}c_{y}\in T=T_{\eps}^{\eps}(z)$. Then $\langle x,y\rangle=\Gamma g b_{s_x}c_{u_x}a_vb_s=\Gamma g b_{s_y}c_{u_y}c_u$ 
	%with  
	%\[v=-2\ln(1-u_x(s_y-s_x)),\ \
	%s=(s_y-s_x)(1-u_x(s_y-s_x)),
	%\ \ u=\frac{u_x}{1-s_x(u_y-u_x)}-s_y.\]
	%Write  $u_x=\frac{u_x'}{1-s_x u_x'}$ for $u_x'\in [-\eps,\eps]$.
	%Then  \[ u_y+u
	%=\frac{\frac{u_x'}{1-s_x u_x'}}{1-(s_y-s_x) \frac{u_x'}{1-s_x u_x'}  }
	%=\frac{u_x'}{1-s_yu_x'}\] implies (b). 
\end{proof}

The following result is proved by the same manner as the previous theorem.
\begin{proposition}[Rectangle] Let ${\tt s},{\tt u}\in (0,1)$ be such that $\max\{\frac{\tt s}{1-{\tt s u}},\frac{\tt u}{1-{\tt s u}}\}<\delta_*/4$ and $z\in X$. The sets
	\begin{eqnarray*}
		S_{{\tt s}}^{{\tt u}}(z)
		&=&
		\Big\{\Gamma g c_u b_s\,:\, u\in [-{\tt u},{\tt u}],\, s=\frac{ s'}{1-us'}\quad \mbox{for some}\quad s'\in [-{\tt s},{\tt s}] \Big\},\\
		T_{{\tt u}}^{{\tt s}}(z)
		&=&
		\Big\{\Gamma g b_s c_u\,:\, s\in [-{\tt s},{\tt s}],\, u=\frac{ u'}{1-su'}\quad \mbox{for some}\quad u'\in [-{\tt u},{\tt u}] \Big\}
	\end{eqnarray*}
	are rectangles, where $g\in\PSL(2,\R)$ such that $z=\Gamma g$.
\end{proposition}
\begin{corollary} Let $\eps\in(0,1)$ and $\tau\in\R$ be such that $\frac{\eps e^{|\tau|}}{1-\eps^2}<\delta_*/4$. 
	Then $\varphi_\tau(S_\eps(z))$ and $\varphi_\tau(T_\eps(z))$ are rectangles. More precisely,
	$\varphi_\tau(S_\eps(z))=S_{\eps e^{-\tau}}^{\eps e^\tau} (z)$ whereas $\varphi_\tau(T_\eps(z))=T_{\eps e^{\tau}}^{\eps e^{-\tau}} (z)$.
\end{corollary}

\begin{remark}\rm \label{rmrec}
	
	%	(a) For  $\tau\in\R$ such that $\min\{\epse^{|\tau|},\epse^{|\tau|}, \frac{\epse^{|\tau|}}{1-\eps\eps}, \frac{\epse^{|\tau|}}{1-\eps\eps} \}<\eps/4$, we have  $\varphi_\tau(S^\eps_\eps(z))=S^{\epse^{\tau}}_{\epse^{-\tau}}(z)$
	%	and $\varphi_\tau(T^\eps_\eps(z))=T^{\epse^{-\tau}}_{\epse^{\tau}}(z)$.

	It is easy to check that
	\[S_\eps(z)=
	\Big\{\Gamma g  c_ub_s: u\in [-\eps,\eps],\, s\in\Big[\frac{-\eps}{1+\eps u},\frac{\eps}{1+\eps u}\Big]\Big\}
	\subset P_{\frac{\eps}{1-\eps^2}}(z). \]
	Therefore, for a given local cross section $P_{\rho}(z)$, 
	any rectangle $S_\eps(z)$ with $\frac{\eps}{1-\eps^2}<\rho$ has the following property:
	\begin{enumerate}
		\item[(a)] $S_\eps (z)$ is closed and contained in $P_\rho(z)$;
		\item[(b)] $\int S_\eps(z)	= \big\{\Gamma g  c_u b_s: s\in (-\eps,\eps), s=\frac{ s'}{1-us'}\quad \mbox{for some}\quad s'\in (-\eps,\eps) \big\}$.
	\end{enumerate}
	
	Similar properties also hold for $T_\eps(z)$. {\hfill $\Diamond$}
\end{remark}

Each version of rectangles has its own special properties. In this paper we will use both of them.
The following result is a relation between the two versions.
\begin{lemma}\label{rec-pro-lm} Let ${\tt u},{\tt s}\in(0,1)$ be such that $\max\{\frac{\tt s}{1-{\tt s u}},\frac{\tt u}{1-{\tt s u}}\}<\delta_*/4$ and $z\in X$.
	Let $S,T$ be local cross sections and let $S_{\tt s}^{\tt u}(z)\subset S, T_{\tt s}^{\tt u}(z)\subset T$
	be rectangles defined in Proposition \ref{recRthm} such that $\pr_{T}(S_{\tt s}^{\tt u}(z))$ and $\pr_{S}(T_{\tt s}^{\tt u}(z))$ are well-defined. Then 
	\begin{equation} \pr_{T}(S_{\tt s}^{\tt u}(z)) =T_{\tt u}^{\tt s}(z)\quad \mbox{and}\quad 
		\pr_{S}(T_{\tt s}^{\tt u}(z))= S_{\tt u}^{\tt s}(z).
	\end{equation} 
	$S_{\tt s}^{\tt u}(z)$ is the projection of $T_{\tt s}^{\tt u}(z)$ on $S$ and $T_{\tt s}^{\tt u}(z)$ is the projection of $S_{\tt s}^{\tt u}(z)$ on $T$.
\end{lemma}
\begin{proof}
	For $x=\Gamma g c_u b_s\in S_{\tt s}^{\tt u}(z)$, we write 
	$x=\Gamma g b_{\tilde s}c_{\tilde u} a_{\tilde t}$ for
	\[ \tilde s=\frac{s}{1+us}, \quad\tilde  u=u(1+us),\quad\tilde  t=-2\ln(1+us).\]
	By the definition of $S_{\tt s}^{\tt u}(z)$, $s=\frac{s'}{1-us'}$ for some $s'\in [-{\tt s},{\tt s}]$. This implies
	that $\tilde s=s'\in [-{\tt s},{\tt s}]$. In addition, $1-\tilde s u=1-\frac{s}{1+us}u=\frac{1}{1+us}$ yields 
	$\tilde u=u(1+us)=\frac{u}{1-\tilde s u}$; hence  $\tilde x=\varphi_{-\tilde t}( x)=\pr_T(x)=\Gamma g b_{\tilde s}c_{\tilde u}\in T_{\tt u}^{\tt s}(z)$ shows that $\pr_T(S_{\tt s}^{\tt u}(z))\subset T_{\tt u}^{\tt s}(z)$. Conversely, if $\tilde x=\Gamma g b_{\tilde s}c_{\tilde u}\in T_{\tt u}^{\tt s}(z)$,
	then
	$\tilde x=\Gamma g c_ub_sa_t$ for 
	\[u=\frac{\tilde u}{1+\tilde u\tilde s},\quad s= \tilde s(1+\tilde u\tilde s), \quad t =2\ln (1+\tilde u\tilde s).\]
	Similarly, we can check that $u\in [-{\tt u},{\tt u}]$ and $s =\frac{\tilde s}{1-u\tilde s}$
	for $\tilde s\in [-{\tt s},{\tt s}]$. Set $x=\Gamma gc_ub_s\in S_{\tt s}^{\tt u}$ 
	to get $x\in S_{\tt s}^{\tt u}$ and $\tilde x=\varphi_{t}(x)=\pr_{T}(x)$, which verifies 
	$T_{\tt u}^{\tt s}(z)\subset \pr_T(S_{\tt s}^{\tt u}(z)) $.  The latter can be proved analogously. 
\end{proof}

Let  $R$ be a rectangle and $x\in R$. We define
\begin{equation}\label{WxT}
	W^s(x,R)=\{\langle x,y\rangle_R, y\in R\}\subset R \quad\mbox{and}\quad
	W^u(x,R)=\{\langle y,x\rangle_R, y\in R \}\subset R.
\end{equation} 
The next result provides  precise forms for $W^s(x,R)$ and $W^u(x,R)$ in the cases $R=S_\eps(z)$
and $R=T_\eps(z)$.  
\begin{proposition}\label{WSlm} (a) Let $S_\eps(z)$ be defined in Proposition \ref{recRthm}. 
	Let $z=\Gamma g$ and $x=\Gamma gc_{u_x}b_{s_x}\in S_\eps(z)$, where $s_x=\frac{s_x'}{1-u_xs_x'}$ for some $s_x'\in [-\eps,\eps]$. Then
	\begin{eqnarray}\label{WsxS}
		W^s(x,S_\eps(z))&=&\Big\{\Gamma gc_{u_x} b_s\,:\,  s=\frac{s'}{1-u_xs'}\quad\mbox{for some} \quad s'\in [-\eps,\eps] \Big\},\\
		\label{WuxS} 
		W^u(x,S_\eps(z))&=&\Big\{\Gamma g c_ub_s\,:\, u\in [-\eps,\eps], \ s=\frac{s'_x}{1-u s'_x}\Big \}.  
	\end{eqnarray}
	(b) Let $T_\eps(z)$ be defined in Proposition \ref{recRthm}. 
	Suppose $z=\Gamma g$ and $x=\Gamma gb_{s_x}c_{u_x}\in T_\eps(z)$, where $u_x=\frac{u_x'}{1-s_xu_x'}$ for some $s_x'\in[-\eps,\eps]$. Then
	\begin{eqnarray}
		W^s(x,T_\eps(z))&=&\Big\{\Gamma gb_s c_u\,: \, s\in [-\eps,\eps],\ u=\frac{u'_x}{1-s u'_x}\Big \},  \\
		W^u(x,T_\eps(z))&=&\Big\{\Gamma gb_{s_x} c_u\,:\,  u=\frac{u'}{1-s_xu'}\quad\mbox{for some} \quad u'\in [-\eps,\eps] \Big\}. \label{WusT}
	\end{eqnarray}
\end{proposition}
\begin{proof} We prove (a) only. If $w\in W^s(x,S_\eps(z))$, then $w=\langle x, y\rangle_{S_\eps(z)}$
	for some $y=\Gamma g c_{u_y}b_{s_y}\in S_\eps(z)$. By Lemma \ref{lcexpl}, $w=\Gamma g c_{u_w}b_{s_w}$ with $u_w=u_x$ and $s_w=\frac{s_y'}{1-u_xs_y'},$
	where $s_y'\in [-\eps,\eps]$ satisfies $s_y=\frac{s_y'}{1-u_ys_y'}$. 
	Conversely, if $v=\Gamma g c_{u_x}b_{s_v}$ for $s_v=\frac{s_v'}{1-u_xs_v'}$, $s_v\in[-\eps,\eps]$, then $v=\langle x, y\rangle$ with $y=\Gamma g c_{u_y}b_{s_y}$ for $u_y\in[-\eps,\eps]$
	and $s_y=\frac{s_v'}{1-u_ys_v'}$; hence $v\in W^s(x,S_\eps(z))$ and we have \eqref{WsxS}. 
	The technique is similar for \eqref{WuxS}. 
	%
	%(b) For $y=\Gamma g b_{s_y}c_{u_y}\in T_\eps(z)$ and $w=\langle x,y\rangle$. Then by the proof of Theorem \ref{recRthm},
	%$w=\Gamma g b_{s_y}c_{\bar u}$ with
	%$\bar u=\frac{u_x'}{1-s_yu_x'}$. Note that $u_x'=\frac{u_x}{1+s_xu_x}$ implies that
	%	\[\bar u=\frac{\frac{u_x}{1+s_xu_x}}{1-s_y\frac{u_x}{1+s_xu_x}}=\frac{u_x}{1+(s_x-s_y)u_x}.\]
	%Since $y=\Gamma g b_{s_y}c_{u_y}$ is arbitrary in $T_\eps(z)$, we have \eqref{WsxT}.
	%Similarly, for $w=\langle x, y\rangle \in W^u(x,T)$ with $y=\Gamma g b_{s_y}c_{u_y}$,
	%then $w=\Gamma g b_{s_x} c_{\bar u}$ with $\bar u=\frac{u_y'}{1-s_xu_y'}$. Note that 
	%$u_y'\in[-\eps,\eps]$ satisfying $u_y'=\frac{u_y}{1+u_ys_y}$. This implies \eqref{WuxT}. 
\end{proof}
The following results follow directly from the previous proposition.

\begin{corollary}\label{W=rm}
	With the setting in Proposition \ref{WSlm}, the following statements hold.
	\begin{itemize}
		\item[(a)] Let $x=(u_x,s_x)_z, y=(u_y, s_y)_z\in S_\eps(z)$. 
		Then $W^s(x,S_\eps(z))=W^s(y,S_\eps(z))$ if and only if $u_x=u_y$, whereas $W^u(x,S_\eps(z))=W^u(y,S_\eps(z))$ if and only if $s_x'=s_y'$.
		\item [(b)] Let $x= (s_x,u_x)'_z, y=(s_y,u_y)'_z\in T_\eps(z)$.
		Then $W^s(x,T_\eps(z))=W^s(y,T_\eps(z))$ if and only if $u_x'=u_y'$, whereas $W^u(x,T_\eps(z))=W^u(y,T_\eps(z))$ if and only if
		$s_x=s_y $. 
		
	\end{itemize}
\end{corollary}

\begin{corollary} One has
	\[W^s(x, S_\eps(z))=W^s(x)\cap S_\eps(z),\quad W^u(x,T_\eps(z))=W^u(x)\cap T_\eps(z).\]
\end{corollary}

The next result is another relation between the two versions of rectangles.
\begin{proposition}\label{proa}
	With the setting in Lemma \ref{rec-pro-lm}, for any $x\in S_\eps(z)$ and $w\in T_\eps(z)$, one has
	\begin{enumerate}
		\item[(a)] $\pr_T W^a(x,S_\eps(z))=W^a(\pr_T(x),T_\eps(z))$ \ for \  $a=u, s$;
		\item[(b)] $\pr_S W^a(w,T_\eps(z))=W^a(\pr_S(w),S_\eps(z))$ \ for \ $a=u, s$. 
		
	\end{enumerate}
\end{proposition}
\begin{proof} (a) Let $z=\Gamma g$ for $g\in \PSL(2,\R)$ and $x=\Gamma gc_{u_x}b_{s_x}$ for $s_x=\frac{s_x'}{1-u_xs_x'}$ with some
	$s_x'\in[-\eps,\eps]$. Using the proof of Lemma \ref{rec-pro-lm}, we have
	$\tilde x=:\pr_T(x)=\Gamma g b_{s_{\tilde x}}c_{u_{\tilde x}} \in T_\eps(z)$ with  $s_{\tilde x}=s_x', u_{\tilde x}=\frac{u_x}{1-s_{\tilde x} u_x}$. 
	If $v\in W^s(x,S_\eps(z))$, then according to \eqref{WsxS},
	$v=\Gamma g c_{u_x}b_{s_v}$ implies $\tilde v:=\pr_{T}(v)=\Gamma g b_{s_{\tilde v}}c_{u_{\tilde v}}$
	with $s_{\tilde v}=s_v'$, $u_{\tilde v}=\frac{u_x}{1-s_{\tilde v}u_x}$. 
	This yields $\tilde v\in W^s(\tilde x,T_\eps(z))$ by Corollary \ref{W=rm}
	and hence $\pr_T W^s(x,S_\eps(z))\subset W^s(\pr_T(x),T_\eps(z))$.
	On the other hand, if  $y\in W^s(\pr_T(x),T_\eps(z))$ then $y=\Gamma g b_{s_y}c_{u_y}$
	with $u_y'=u_{\tilde x}'=u_x$. Setting $v=\Gamma g c_{u_x}s_{v}\in W^s(x,S_\eps(z)) $ with $s_v=\frac{s_y}{1-u_xs_y}$, 
	we obtain $y=\pr_S(v)$ due to the proof of Lemma \ref{rec-pro-lm}. As a result, $ W^s(\pr_T(x),T_\eps(z))\subset \pr_T W^s(x,S_\eps(z))$, proving
	$\pr_T W^s(x,S_\eps(z))=W^s(\pr_T(x),T_\eps(z))$. 
	
	Next, for $y\in W^u(x, S_\eps(z))$, $y=\Gamma g {c_{u_y}}b_{s_y}$ with $s_y'=s_x'$ by Corollary \ref{W=rm}. 
	Then $\pr_{T}(y)=\Gamma g b_{s_{\tilde y}}c_{u_{\tilde y}}$ with $s_{\tilde y}=s_x'=s_{\tilde x}$
	implies  $\pr_T(y)\in W^u(\pr_T(x),T_\eps(z))$. Conversely, 
	if $\tilde w\in W^u(\pr_T(x),T_\eps(z))$, then $\tilde w=\Gamma gb_{s_{\tilde x}}c_{u_{\tilde w}}$.
	Define $w=\Gamma g c_{u_w}b_{s_w}$ for $u_w=\frac{u_{\tilde w}}{1+u_{\tilde w}{s_{\tilde w}}}$ and $ s_w'=s_{\tilde x}$ to have $\pr_T(w)=\tilde w$.
	Also,  $s_w'=s_x'$
	yields  $w\in W^u(x,S_\eps(z))$, which completes the proof of (a). Statement (b) follows from (a). 
\end{proof}

The next result is helpful afterwards.
\begin{lemma}\label{=x=ylm} Let $R$ be a rectangle and $x,y,z,w\in R$. Then
	
	\noindent(a) $\big\langle\langle x, y\rangle_R,z\big\rangle_R=\big\langle x,\langle y,z\rangle_R\big\rangle_R=\langle x, z\rangle_R$;
	
	\noindent(b) if $y\in W^s(x,R)$, then $\langle x,y\rangle_R=y$;
	
	\noindent (c) if $y\in W^u(x,R)$, then $\langle y,x\rangle_R=y$;
	
	\noindent (d) $\big\langle  \langle x,y\rangle_R, \langle z,w\rangle_R\big\rangle_R=\langle x, w\rangle_R$.
\end{lemma}
\begin{proof}
	(a) The proof is similar to Lemma \ref{xyz} (a). 
	
	\noindent(b) Assume that $y\in W^s(x,R)$. Then $y=\langle x, y'\rangle_R$
	for some $y'\in R$ implies that $\big\langle x,y\rangle_R=\langle x,\langle x, y'\rangle_R\big\rangle_R= \langle x, y'\rangle_R=y$ by (a).
	
	\noindent 
	(c) The manner is similar to (a).

	\noindent (d) Using (a)-(c), we have
	$ \big\langle \langle y,x\rangle_R, \langle z,y\rangle_R\big\rangle_R
	=\big\langle  y,\langle x, \langle z,y\rangle_R\rangle_R\big\rangle_R
	=\big\langle  y,\langle x,y\rangle_R\big\rangle_R=\langle y,y\rangle_R=y$.
	\end{proof}

\section{Expansivity}\label{expansection} 
In this section we study a nice property of hyperbolic dynamical systems, named expansivity.
Roughly speaking, 
for more variation of expansivities, the reader can 
if two orbits of the flow are close enough for the whole time then they must be identical. 
\begin{definition}[\cite{bw}] \rm
	Let $(M,d)$ be a compact metric space.  
	A continuous flow $\phi_t:M\longrightarrow M$ is called {\em expansive} if for each $\eps>0$ there exists $\delta=\delta(\eps)>0$ with the following property. 
	If $s:\R\rightarrow \R$ is a continuous function with $s(0)=0$ and 
	\[d(\phi_t(x),\phi_{s(t)}(y))<\delta \quad \mbox{for all}\quad t\in \R,\]
	then $y=\phi_\tau(x)$ for some $\tau\in(-\eps,\eps)$.
\end{definition}
The next result was initially introduced  in \cite{bo-symb} to prove the expansivity of general hyperbolic flows. 
Expansivity of the flow $(\varphi_t)_{t\in \R}$ was reproved in \cite{Hien} by a new approach, using the injectivity radius. 
\begin{theorem}[\cite{bo-symb}]\label{expan-thm}
	For each $\eps>0$ there is a $\delta=\delta(\eps)>0$ with the following property.
	If $x,y\in X$, $L>0$ and $s:\R \to \R$ continuous with $s(0)=0$ satisfy
	\begin{equation}\label{tst}
		d_X(\varphi_t(x),\varphi_{s(t)}(y))\leq \delta\quad \mbox{for all}\quad t\in[-L,L],
	\end{equation} 
	then \begin{equation}\label{sttt}
		|s(t)-t|\leq \eps\quad\mbox{for all}\quad t\in[-L, L].
	\end{equation} 
	Furthermore, let $w=\langle x,y\rangle=W^s_{\eps}(\varphi_v(x))\cap W^u_\eps(y)$ for appropriate $v\in (-\eps,\eps)$ in Corollary \ref{lpslm}.
	Then 
	\begin{equation}\label{sL}
		d_X(\varphi_{t}(w),\varphi_{t}(x))<2\eps \quad \mbox{for all}\quad t\in [-L,L],
	\end{equation} 
	\begin{equation}\label{uL}d_X(\varphi_t(w),\varphi_t(y))<3\eps\quad\mbox{for all}\quad t\in [-L,L],
	\end{equation}
	and
	\begin{equation}\label{exp}
		d_X(y,\varphi_v(x))< 2\eps e^{-L}.
	\end{equation} 
	In particular, the flow $(\varphi_t)_{t\in\R}$ is expansive.
\end{theorem}
\begin{proof}
	We follow the proof of Theorem 3.2 in \cite{Hien} for the first part.  
	Let $\eps>0$ be given and $\rho=\rho(\eps)$  as in Lemma \ref{eps0}.
	Let $\delta_1=\delta_1(\rho)$ be as in Corollary \ref{lpslm}. 
	Let $\delta_2=\delta_2(\eps_1)$ be as in Lemma \ref{konvexa}, where $\eps_1=\frac{e^{\rho/2}-1}{e^{\rho/2}+1}$. We define $\delta=\min\{\delta_1,\delta_2\}$. 
	
	\underline{Step 1:} Proof of \eqref{sttt}.
	Write $x=\Gamma g, y=\Gamma h$ for $g,h\in\PSL(2,\R)$ and fix $L>0$. 
	For each $t\in [-L,L]$, there is $\gamma(t)\in\Gamma$ so that
	\begin{eqnarray}\label{gammaL}
		d_X(\varphi_{s(t)}(y),\varphi_t(x))
		=d_X(\Gamma ha_{s(t)}, \Gamma ga_t)
		=d_\G(\gamma(t) ha_{s(t)},ga_t)<\delta.
	\end{eqnarray}
	It was shown in the proof of \cite[Theorem 3.2]{Hien} that 
	\begin{equation*} \gamma(t)=\gamma(0)\quad\mbox{ for all }\quad t\in [-L,L].
	\end{equation*} 
	Setting $\gamma_0=\gamma(0)$, \eqref{gammaL} becomes
	%We first verify that 
	%\begin{equation} \gamma(t)=\gamma(0)\quad\mbox{ for all }\quad t\in [-L,L].
	%\end{equation} 
	%Indeed, since $s:[-L,L]\to\R$ is uniformly continuous, there is $0<\rho=\rho(L,\delta)<\delta$
	%such that if $t_1,t_2\in [-L,L]$ and $|t_1-t_2|<\rho$ then 
	%$|s(t_1)-s(t_2)|<\delta$.
	%Denote $c_1(t)=ga_t$ and $c_2(t)=ha_{s(t)}$.
	%For any $t_1,t_2\in\R$, we have
	%\begin{eqnarray*}
	%	d_\G(\gamma(t_2)^{-1}\gamma(t_1)c_1(t_1),c_1(t_1))
	%	&\leq& d_\G(\gamma(t_1)c_1(t_1),\gamma(t_2)c_1(t_1))
	%	\\
	%	&\leq& d_\G(\gamma(t_1)c_1(t_1),c_2(t_1))
	%	+d_\G(c_2(t_1),c_2(t_2))
	%	\\ &&
	%	+\ d_\G(c_2(t_2),\gamma(t_2)c_1(t_2))
	%	\\
	%	&&+\ d_\G(\gamma(t_2)c_1(t_2),\gamma(t_2)c_1(t_1))
	%	\\
	%	&\leq& 2\delta+\frac{1}{\sqrt 2}|s(t_1)-s(t_2)|+|t_1-t_2|
	%	\\
	%	&\leq& 2\delta +|s(t_1)-s(t_2)|+|t_1-t_2|.
	%\end{eqnarray*}
	%Here if we specialize this to $t_1=0$ and $t_2\in [0,\rho]$,
	%then $|t_1-t_2|<\rho$ implies $|s(t_1)-s(t_2)|<\delta$ and 
	%\[d_\G(\gamma(t_2)^{-1}\gamma(t_1)c_1(t_1),c_1(t_1))<4\delta <\sigma_*\]
	%From the property of $\sigma_*$, it follows that 
	%$\gamma(t_2)=\gamma(0)$ for all $t_2\in [0,\rho]$. Then we repeat the argument for
	%$t_1=\rho$ and $t_2\in [\rho, 2\rho]$, we deduce that $\gamma(t)=\gamma(0)$ for all $t\in [0,2\rho]$, which upon further iteration lead to $\gamma(t)=\gamma(0)$ for all $t\in [0,L]$
	%and similarly $\gamma(t)=\gamma(0)$ for all $t\in [-L,0]$. Put $\gamma_0=\gamma(0)\in\Gamma$,
	%then \eqref{gammaL} shows that 
	\begin{equation}\label{AGA} d_\G(a_{-t}g^{-1}\gamma_0 ha_{s(t)},e)=d_\G(\gamma_0 ha_{s(t)},ga_t)<\delta\quad \mbox{for all}\quad t\in [-L,L].
	\end{equation} 
	Write $g^{-1}\gamma_0 h=\pi(G)$ for $G=\big({\scriptsize\begin{array}{cc} a&b\\c&d\end{array}}\big)$ and
	\[A_{-t}GA_{s(t)}=\bigg(\begin{array}{cc} ae^{\frac{s(t)-t}2}&be^{-\frac{s(t)+t}2}\\c e^{\frac{s(t)+t}2}&d e^{\frac{t-s(t)}2}\end{array}\bigg).\]
	Using Lemma \ref{konvexa}, \eqref{AGA} implies that
	\[||a|e^{\frac{s(t)-t}2}-1|+||d| e^{\frac{t-s(t)}2}-1|<\eps_1\quad \mbox{for all}\quad |t|\leq L,\]
	or equivalently 
	\begin{eqnarray}\label{ad} 
		1-\eps_1\leq |a|e^{\frac{s(t)-t}2}\leq 1+\eps_1\quad \mbox{and}\quad 1-\eps_1\leq |d|e^{\frac{t-s(t)}2}\leq 1+\eps_1
		\quad \mbox{for all} \quad |t|\leq L.
	\end{eqnarray}
	Suppose, on the contrary, that $s(t)-t>\eps$ for some $|t|\leq L$. Then 
	$|a| e^{\frac{s(t)-t}2}>(1-\eps_1) e^{\frac{\rho}2}=(1-\eps_1)\frac{1+\eps_1}{1-\eps_1}=1+\eps_1$, which contradicts  \eqref{ad}.
	Therefore  $s(t)-t<\rho$ for all $|t|\leq L$. Similarly, $t-s(t)<\rho$ for all $|t|\leq L$, so
	\begin{equation}\label{strho}|s(t)-t|<\rho\quad \mbox{for all}\quad |t|\leq L. 
	\end{equation}
	Since $\rho<\eps$, we have  \eqref{sttt}. 
	
	\underline{Step 2:} Proof of \eqref{uL}-\eqref{sL}.  Recall from Corollary \ref{lpslm}  that there are $s,u,v\in [-\rho,\rho]$ such that 
	$w=\langle x,y\rangle=\Gamma ga_vb_{s}=\Gamma h c_{u}$. Then for $t\geq 0$, using lemmas \ref{at} and \ref{tslm}, we get
	\begin{align} \notag 
		d_X(\varphi_t(w),\varphi_t(x))
		&\leq  d_\G(ga_vb_{s}a_t,ga_t)
		=d_\G(a_vb_{se^{-t}},e)\\ \notag 
		&= d_\G(b_{se^{-t}}, a_{-v})
		\leq d_\G(b_{se^{-t}}, e)+ d_\G(a_v, e)\\
		&\leq \frac{1}{\sqrt 2}|v|+|s|e^{-t}<2\rho. \label{eq1}
	\end{align} 
	Together with \eqref{strho} this implies that for $t\in [0,L]$
	\begin{eqnarray}
		d_X(\varphi_t(w),\varphi_t(y)) \notag 
		&\leq& 
		d_X(\varphi_t(w),\varphi_t(x))
		+d_X(\varphi_t(x),\varphi_{s(t)}(y))
		+d_X(\varphi_{s(t)}(y),\varphi_t(y))
		\\ \notag 
		&\leq & 
		\frac{1}{\sqrt 2}|v|+|s|e^{-t}+\delta+\frac{1}{\sqrt 2}|s(t)-t|
		\\ \label{eq2}
		&<&\frac{1}{\sqrt 2}\rho +\rho+\delta+\frac{1}{\sqrt 2}\rho<3\rho. 
	\end{eqnarray}
	Furthermore, $w\in W^u_\rho(y)$ yields
	\begin{equation}\label{eq4}
		d_X(\varphi_t(y),\varphi_t(w))<\rho e^{-t}\quad\mbox{for all}\quad t<0.
	\end{equation}
	In conjunction with \eqref{eq2} and $\rho<\eps$ this proves \eqref{sL}.
	Analogously, it follows from \eqref{eq4} and \eqref{tst} that for $t\in [-L,0]$
	\begin{eqnarray*}
		d_X(\varphi_t(w),\varphi_t(x))
		&\leq&
		d_X(\varphi_t(w),\varphi_t(y))
		+d_X(\varphi_t(y),\varphi_{s(t)}(x))
		+d_X(\varphi_{s(t)}(x),\varphi_t(x))
		\\
		&\leq &|u|e^{t}+\delta+\frac{1}{\sqrt 2}|s(t)-t|
		<\rho+\delta+\frac{1}{\sqrt 2}\rho<2\rho. 
	\end{eqnarray*}
	As a consequence,  the statement  \eqref{uL} is proved, using  \eqref{eq1}. 
	
	\underline{Step 3:} Proof of \eqref{exp}. Now, define $\tilde x=\Gamma g a_v$. Then 
	\begin{align*}
		d_X(\varphi_t(w),\varphi_t(\tilde x))
		&\leq  d_X(\varphi_t(w),\varphi_t(x))+ d_X(\varphi_t(x),\varphi_t(\tilde x))
		\\
		&\leq 2\rho+\frac{1}{\sqrt 2} |v| <3\rho \ \mbox{for all}\ t\in [-L,0],
	\end{align*} 
	owing to \eqref{eq1}. It follows from Lemma \ref{eps0}\,(b) that $|s|<\eps e^{-L}$; recall that $w=\Gamma g a_v b_s$.
	Also, using $w=\Gamma hc_u$, $y=\Gamma g$, \eqref{eq2} and Lemma \ref{eps0}\,(a), we get $|u|< \eps e^{-L}$. 
	Now, due to $w=\Gamma ga_vb_s=\Gamma h c_u$,
	\begin{eqnarray*}
		d_X(y,\varphi_v(x))
		&=&d_X(\Gamma h,\Gamma ga_v )=d_X(\Gamma ga_vb_sc_{-u},\Gamma g a_v)
		\\
		&\leq&|s|+|u|< 2\eps e^{-L},
	\end{eqnarray*}
	which is \eqref{exp}.
	Finally, let $L\to \infty$ to have   $y=\varphi_v(x)$, which shows the expansivity of the flow $(\varphi_t)_{t\in\R}$. 
	The proof is complete.
\end{proof} 
Now, we use the expansivity to prove the following auxiliary result, which was introduced in \cite{bo-symb} without a proof. This result will be used several times in Section \ref{Markovsection}.

\begin{lemma}\label{DD'}
	Let $\eps\in(0,\sigma_*/6)$ and  $D=P_\eps(z)$  and $D'=P_\eps(z')$.
	There exists $\delta=\delta(\eps)>0$ with the following property.
	Suppose that	$x,y\in D$, $\langle x,y\rangle_D$ exists and
	there is a continuous function
	$s:[0,T]\rightarrow \R$ with $s(0)=0$ so that
	\[d_X(\varphi_t(x),\varphi_{s(t)}(y))\leq\delta\quad\mbox{for all}\quad t\in [0,T],\]
	$ \varphi_T(x),\varphi_{s(T)}(y)\in D'$ and
	$\langle \varphi_T(x),\varphi_{s(T)}(y)\rangle_{D'}$
	exists.
	Then
	\begin{equation}\label{D=D'}
		\varphi_T(\langle x,y\rangle_D)=\langle\varphi_T(x),\varphi_{s(T)}(y)\rangle_{D'}.
	\end{equation}
\end{lemma}
\noindent
\begin{proof} Let $\eps\in(0,\sigma_*/6)$ be given and take  $\delta=\delta(\eps/3)$ as in  Theorem \ref{expan-thm} to get
	\begin{equation}
		|s(t)-t|\leq \frac{\eps}{3}\quad\mbox{for all}\quad t\in [0,T].
	\end{equation}
	Write $x=\Gamma g, y=\Gamma h$ for $h,g\in\G=\PSL(2,\R)$ such that $d_X(x,y)=d_\G(h,g)$. 
	According to Lemma \ref{prodstr}, 
	\[\langle x,y\rangle_D =\Gamma g b_{s_1e^{v_1}}=\Gamma hc_{u_1}a_{-v_1}\]
	with $s_1=bd, u_1=-\frac{c}d, v_1=-2\ln d$; here $\big({\scriptsize\begin{array}{cc}a&b\\c&d\end{array}}\big)=:A\in\SL(2,\R)$
	satisfies $g^{-1}h=\pi(A)$. Then
	\begin{equation}\label{eq1d}\varphi_T(\langle x,y\rangle_D)=\Gamma g b_{s_1e^{v_1}}a_{T}=\Gamma g a_T b_{s_1e^{v_1-T}}.
	\end{equation}
	If \[B=\left(\begin{array}{cc}ae^{\frac{s(T)-T}{2}}&be^{-\frac{s(T)+T}{2}}\\ce^{\frac{s(T)+T}{2}}&de^{-\frac{s(T)-T}{2}}\end{array}\right)\in\SL(2,\R), \]
	then $\pi(B)=(g a_T)^{-1}ha_{s(T)}$. This implies that
	\begin{equation}\label{eq2d}\langle \varphi_T(x),\varphi_{s(T)}(y)\rangle_{D'}=\Gamma ga_Tb_{s_2e^{v_2}}=\Gamma ha_{s(T)}c_{u_2}a_{-v_2} 
	\end{equation}
	for
	\begin{eqnarray*}
		s_2=bd e^{-s(T)}=s_1e^{-s(T)},\ u_2=-\frac{c}{d}e^{s(T)}=u_1e^{s(T)},\ v_2=-2\ln (d e^{-\frac{s(T)-T}{2}})=v_1-T+s(T);
	\end{eqnarray*}
	note that by Theorem \ref{expan-thm},  $|u_1|\leq \eps e^{-T}/3$ implies   $|u_2|<\eps$, so \eqref{eq2d} is well-defined.
	This yields $s_2e^{v_2}=s_1e^{v_1-T}$. By comparison  \eqref{eq1d} and \eqref{eq2d}, we obtain \eqref{D=D'},  completing the proof.
\end{proof}
\begin{remark}\rm The previous lemma is also true for $s:[-T,0]\to \R$. The proof is similar. 
	{\hfill$\Diamond$}
\end{remark}

\section{Construction of Markov partitions}\label{Markovsection}
In this section we give a rigorous construction of Markov partitions.  We will use the forms of rectangles and local cross sections in Section \ref{lcssection} to construct a so-called pre-Markov partition, and then  we
follow Bowen's work in \cite{bo-symb} to construct a Markov partition of arbitrarily small size step by step, in that we even could somewhat simplify \cite{bo-symb}. The special forms of rectangles leads to a more explicit and intuitive Markov partition.

First, we introduce the notion of `proper family'.
\begin{definition}[Proper family]\label{pfdf} \rm 
	Let $\alpha>0$ be given and let ${\mathscr T}=\{T_1,\dots,T_n\}$ be a family of closed sets in $X$.  We call $\mathscr T$ a {\em proper family of size} $\alpha$ if
	
	(i) $X=\varphi_{[-\alpha,0]}(\bigcup_{i=1}^nT_i)$;\\
	there is a family of differential local cross sections ${\mathscr D}=\{D_1,\dots,D_n\}$ such that
	
	(ii) $\diam D_i<\alpha$;
	
	(iii) $T_i\subset \int D_i$;
	
	(iv) for $i\ne j$, at least one of the sets
	$D_i\cap \varphi_{[0,\alpha]}(D_j)$ and $D_j\cap \varphi_{[0,\alpha]}(D_i)$
	is empty.
	
\end{definition}
In particular, it follows from (iv) that if $i\ne j$, then $D_i\cap D_j=\varnothing$. 
%	Let ${\mathcal {T}}$ be a proper family of size $\alpha$  small. 
%	From (ii), it follows that for any $x\in \T$, there is $t(x)$ a first positive time
%	$t(x)\leq \alpha$ so that $\varphi_{t(x)}(x)\in \T$.
%	We denote $t(x)$ by $t_x$. 

\begin{definition}[Poincar\'e map]\rm
	Let ${\mathscr T}=\{T_1,\dots,T_n\}$ be a proper family. For any $x\in \T=T_1\cup\dots\cup T_n$, denote by 
	$t(x)$ the first return time, which is the smallest $t>0$ such that $\varphi_{t}(x)\in \T$. 
	The map $\P_{\mathscr T}:  \T \longrightarrow \T$ defined by
	\[\P_{\mathscr T}(x)=\varphi_{t(x)}(x)\]
	is called the {\em Poincar\'e map} with respect to the family ${\mathscr T}$.  
\end{definition}
%To simply the notation, we drop the subscript ${\mathcal T}$ in $\P_{\mathcal T}$ if there is no confusion.
%\input{pm.tpx}
The first return time is also strictly bounded from below by a positive number as follows.
%\begin{lemma}\label{beta}
%	There exists $\beta>0$ so that $t(x)\geq \beta$ for all $x\in {\T}$.
%\end{lemma}
%\begin{proof}
%	Note that $T_1,\dots, T_n$ are closed and $T_i\cap T_j=\varnothing$ for $i\ne j$. 
%	Let \[\beta=\min\{{\rm dist}(T_i,T_j):i\ne j\}>0;\]
%	here ${\rm dist}$ denotes the Hausdorff distance.
%	We claim that $t(x)\geq \beta$ for all $x\in\T$.
%	Suppose on the contrary that there is $x\in T_i$ so that $t(x)<\beta$ and $y=\varphi_{t(x)}(x)\in T_j$
%	for some $i\ne j$. Then 
%	\[d_X(x,y)=d_X(x,\varphi_{t(x)}(x))\leq\frac{1}{\sqrt 2}|t(x)|<\beta,\]
%	which contradicts the definition of $\beta$.
%\end{proof}

\begin{proposition}
	The Poincar\'e map $\P_{\mathscr T}:T_1\cup\dots \cup T_n\to  T_1\cup\dots \cup T_n$ is a bijection. 
\end{proposition}
\begin{proof}
	Take $x,y\in \T=T_1\cup\dots\cup T_n$ such that $\P_{\mathscr T}(x)=\P_{\mathscr T}(y)$ or equivalently $\varphi_{t(x)}(x)=\varphi_{t(y)}(y)$. In order to obtain $x=y$, we must show $t(x)=t(y)$.
	Suppose, in a contrary, that $t(x)\ne t(y)$. If $t(x)>t(y)$ then $0<t(x)-t(y)<t(x)$ 
	and $y=\varphi_{t(x)-t(y)}(x)\in {\mathcal T}$, which contradicts the definition of $t(x)$. 
	The same occurs for $t(x)<t(y)$.  Therefore $t(x)=t(y)$ and we deduce that $\P_{\mathscr T}$ is injective.
	Since $({\varphi_t})_{t\in\R}$ is time reversal invariant, $\P_{\mathscr T}$ is surjective, which completes the proof.
\end{proof}

Note that the first return time map $t$ and the Poincar\'e map $\P_{\mathscr T}$ are not continuous on ${\T}$ but they are continuous on 
\[{\T^*}=\big\{x\in {\T}: \P_{\mathscr T}^k(x)\in \int T_1\cup\dots\cup \int T_n\quad \mbox{for all}\quad k\in\Z \big\}.\]
It does not matter since 
$\T^*$ is dense in $\T$ and $$\varphi_\R(\T^*)=\big\{x\in X:(\varphi_\R(x)\cap \T) \subset \int T_1\cup\dots \cup \int T_n \big\}$$
is dense in $X$.

\begin{definition}[Markov partition] \label{MPdn}\rm 
	A proper family ${\mathscr T}=\{T_1,\dots,T_n\}$ is called a {\em Markov partition} if
	each member in $\mathscr T$ is a rectangle and ${\mathscr T}$ satisfies the Markov property:
	
	$(M_s)$  if $x\in U(T_i,T_j)=\overline{\{x\in {\cal T}^*: x\in \int T_i, \P_{\mathscr T}(x)\in \int T_j\}}$,
	then $W^s(x,T_i)\subset U(T_i,T_j)$;
	
	$(M_u)$ if $x\in V(T_i,T_k)=\overline{\{x\in {\cal T}^*: x\in \int T_i, \P_{\mathscr T}^{-1}(x)\in \int T_k \}}$,
	then $W^u(x,T_i)\subset V(T_i,T_k)$. 
	
\end{definition}

\begin{figure}[h]
	\begin{center}
		\begin{minipage}{1\linewidth}
			\centering
			\includegraphics[angle=0,width=0.7\linewidth]{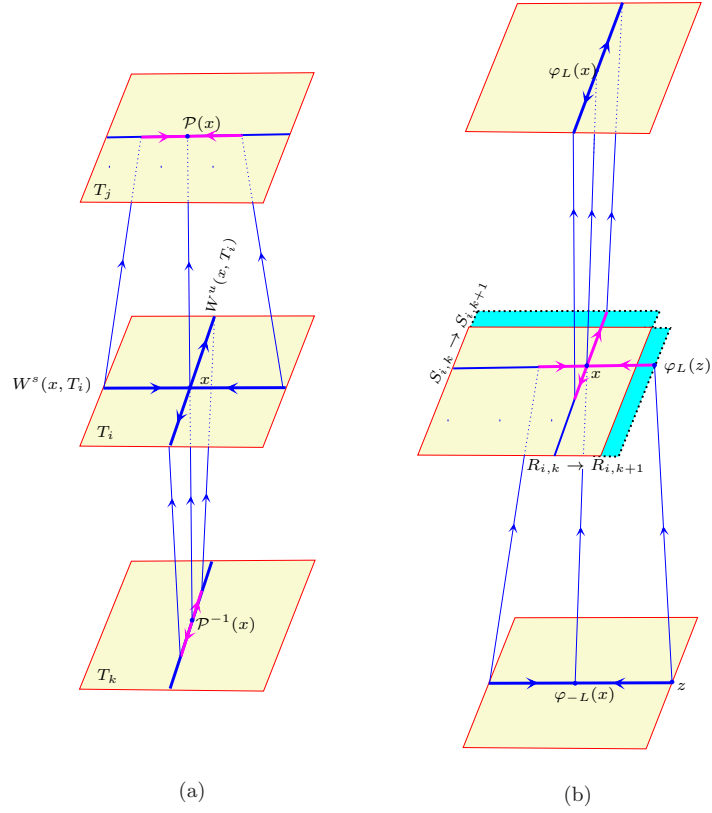}
		\end{minipage}
	\end{center}
	\caption{(a) Markov property
		\quad (b) Enlarge rectangles}\label{ME}
\end{figure} 
\begin{remark}\label{PP-1}\rm   Let  $x\in T_i, \P_{\mathscr T}(x)\in T_j$
	and $z\in W^s(x,T_i)$.  If $\P_{\mathscr T}(x)\in T_j$, then $\P_{\mathscr T}(z)\in W^s(\P_{\mathscr T}(x), T_j)$. For instance, by Lemma \ref{=x=ylm}\,(b) $z=\langle x,z\rangle_{T_i}$. Similarly to Lemma \ref{k=1}, we obtain $\P_{\mathscr T}(z)=\langle \P_{\mathscr T}(x),\P_{\mathscr T}(z)\rangle_{T_j}$, and so $\P_{\mathscr T}(z)\in W^s(\P_{\mathscr T}(x),T_j)$.	
	Analogously, if $y\in T_i$, $z\in W^u(y, T_i)$
	and $\P^{-1}_{\mathscr T}(y),\P_{\mathscr T}^{-1}(z)\in T_k$,
	then  $\P_{\mathscr T}^{-1}(z)\in W^u(\P_{\mathscr T}^{-1}(y),T_k)$; see Figure \ref{ME}\,(a) for an illustration of the Markov property.   
		{\hfill$\Diamond$}
\end{remark}

\begin{proposition}\label{MKallilm} Suppose that   ${\mathscr T}$ is a Markov partition
	and $S_{-N},\dots, S_N\in {\mathscr T}$. Let $x,y\in S_0\cap \T^*$
	and 
	$z=\langle x,y\rangle_{S_0}\in \T^*$. Then
	
	(a) if $\P_{\mathscr T}^i(x)\in S_i$ for $0\leq i\leq N$, then $\P_{\mathscr T}^i(z)\in W^s(\P_{\mathscr T}^i(x),S_i)$ for $0\leq i\leq N$. In particular, 
	$\P_{\mathscr T}^i(z)\in S_i$ for $0\leq i\leq N$;
	
	(b) if $\P_{\mathscr T}^i(y)\in S_i$ for $-N\leq i\leq 0$, then
	$\P_{\mathscr T}^i(z)\in W^u(\P_{\mathscr T}^i(y), S_i)$ for $-N\leq i\leq 0$.
	In particular, 
	$\P_{\mathscr T}^{i}(z)\in S_{-i}$ for $-N\leq i\leq 0$.
\end{proposition}
\begin{proof} (a) By the assumption, it follows that  $\P_{\mathscr T}^i(x)\in U(S_i,S_{i+1})$
	for $0\leq i\leq N-1$. We prove by induction. For $i=1$,
	$z=\langle x,y\rangle_{S_0}\in W^s(x,S_0)$. 
	By property $(M_s)$,  $z\in U(S_0, S_1)$. 
	Due to $z\in \T^*$, $\P_{\mathscr T}(z)\in S_1$. Since $z\in W^s(x,S_0)$,
	it follows from Remark \ref{PP-1} that $\P_{\mathscr T}(z)\in W^s(\P_{\mathscr T}(x), S_1)$,
	so the statement holds for $i=1$. 
	Assume that $\P_{\mathscr T}^i(z)\in W^s(\P_{\mathscr T}^i(x), S_i)$ for $1\leq i\leq N-1$.
	Since $\P_{\mathscr T}^i(x)\in U(S_i,S_{i+1})$ and $P_{\mathscr T}^i(z)\in W^s(\P_{\mathscr T}^i(x),S_i)$,
	it follows that  $\P_{\mathscr T}^i(z)\in U(S_i,S_{i+1})$, and hence $\P_{\mathscr T}^{i+1}(z)\in S_{i+1}$, due to $z\in \T^*$.
	This yields $\P_{\mathscr T}^{i+1}(z)\in W^s(\P_{\mathscr T}^{i+1}(x), S_{i+1})$
	by Remark \ref{PP-1}
	and the conclusion is obtained.
	
	(b) Here the argument is analogous.
\end{proof}

\begin{remark} \rm In geometric meaning, Proposition \ref{MKallilm} says that, if the future orbit $\{\varphi_t(x),t\geq 0\}$ of
	$x\in \int S_{0}$ passes through $\int S_{1}, i=1,2,3,\dots$ (in sequence)
	and the past orbit $\{\varphi_t(y), t< 0\}$ of $y\in\int S_{0}$ passes through $\int S_{i}, i=-1,-2,\dots$ (in sequence)
	then the orbit of $\langle x,y\rangle_{S_{0}}\in\int S_{0}$ has both properties; see Figure \ref{pm} for an illustration.  This property is used as the definition of Markov partitions in \cite{Po_s}.
	{\hfill$\Diamond$}
\end{remark}

\begin{figure}[ht]
	\begin{center}
		\begin{minipage}{1\linewidth}
			\centering
			\includegraphics[angle=0,width=1\linewidth]{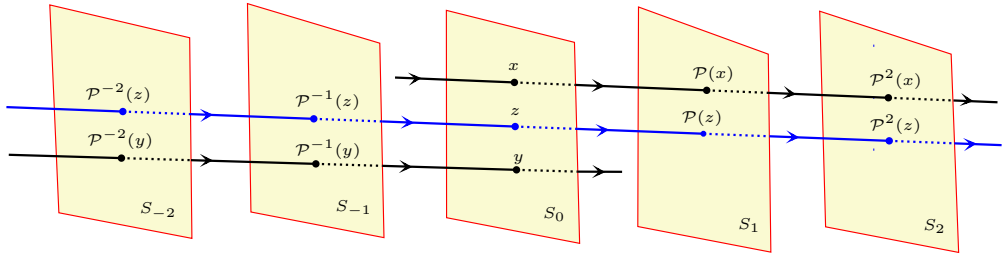}
		\end{minipage}
	\end{center}
	\caption{Markov property: $z=\langle x,y\rangle_{T_{x_0}}$ has properties of both $x$ and $y$}\label{pm}
\end{figure}

In the rest of this paper we prove the following main result.

\begin{theorem}\label{mthm}
	The flow $(\varphi_t)_{t\in\R}$ has a Markov partition of arbitrary small size.
\end{theorem} 

The construction of Markov partitions can be summarized as follows.

\begin{itemize}
	\item For arbitrarily small $\alpha>0$, construct a proper family of size $\alpha$ consisting of rectangles $B_1,\dots,B_n$ rectangles, which contain rectangles $K_1,\dots, K_n$  
	with certain properties (Theorem \ref{step1}).
	\item Enlarge rectangles $K_1,\dots, K_n$ to $C_1,\dots, C_n$ satisfying
	Lemma \ref{CiCj}.
	\item Decompose $C_i$ into smaller sets $E_{ji}^1,E_{ji}^2, E_{ji}^3, E_{ji}^4$ in Lemma \ref{Eji1-4} and define  family ${\mathfrak C}_i$ of sets in $C_i$; see Lemma \ref{frankC}.
	
	\item Construct equivalence classes of
	elements in $C_i$ whose orbits visit the same member of ${\mathfrak C}_1,
	\dots {\mathfrak C}_n$ in the same order for sufficiently large times. 
	\item Prove that after sliding appropriately small times, these equivalence classes are a Markov partition; see lemmas \ref{reclmG} and \ref{flm2}.  
\end{itemize}

Fix $\eps\in (0,\sigma_*/5)$ and define $\delta_1=\delta(\eps)$
from Corollary \ref{lpslm}, and $\delta_2=\delta(\eps)$ as in Lemma \ref{DD'}. 
We define $\delta=\min\{\delta_1/4,\delta_2/4, \sigma_*/6 \}$
and consider $\alpha\in (0,\delta)$.

First, we construct a so-called pre-Markov partition, 
which is stated in \cite{bo-symb} without a proof. 
A similar assertion can be found in \cite{Po87}.

\begin{theorem}\label{step1}
	There are a family of differentiable local cross sections ${\mathscr D}=\{D_1,\dots,D_n\}$ and 
	two families of  rectangles $\mathscr K=\{ K_1,\dots, K_n\}, \mathscr B=\{B_1,\dots,B_n\}$ satisfying
	\begin{enumerate}
		\item[(a)] $K_i\subset \int B_i, B_i\subset \int D_i, i=1,\dots,n$;
		\item [(b)]$ \diam D_i<\alpha, i=1,\dots,n$; 
		\item [(c)] for $i\ne j$, at least one of the sets $D_i\bigcap \varphi_{[0,2\alpha]}(D_j)$ and  $D_j\bigcap \varphi_{[0,2\alpha]}(D_i)$ is empty;
		\item [(d)] $X= \varphi_{[-\alpha,0]}(\bigcup_{i=1}^n\int K_i)= \varphi_{[-\alpha,0]}(\bigcup_{i=1}^n\int B_i)$; %, where $\int B_i$ is the interior of $B_i$ as a subset of metric space $D_i$;
		\item [(e)] if $B_i\bigcap \varphi_{[-\alpha,\alpha]}(B_j)\ne\varnothing$,
		then $B_i\subset \varphi_{[-2\alpha,2\alpha]}(D_j)$.
	\end{enumerate}
\end{theorem} 

%To simply the notation, from now on we drop the subscript $D_i$ in $\int_{D_i}(B_i)$ and in similar cases if there is no confusion.

In comparison with the statement in \cite{bo-symb},
there is a slightly difference of the flow times and the presence of $K_1,\dots,K_n$.  Later in our construction, we will enlarge $K_1,\dots, K_n$ to $C_1,\dots, C_n$, which are still included in $B_1,\dots,B_n$, and conditions (c), (e) will be crucial in proving the Markov property.

\begin{proof} The idea of this proof is carefully modified from that of \cite[Lemma 7]{bw}.

	Note that due to $0<\alpha<\sigma_*/6$,  any Poincar\'e section
	of radius at most  $\lambda$ is a local cross section of time $2\alpha$; see Lemma \ref{lcexpl}. 
	Since $X$ is compact, there are  $x_1,\dots,x_m\in X$ pairwise disjoint such that
	\begin{equation}\label{epsX} X=\varphi_{[-\alpha,0]}(\int S_{\alpha/16}(x_1))\cup\bigcup_{k=2}^m\varphi_{[-\alpha,0]}(\int  P_{\alpha/2}({x_k})).
	\end{equation} 
	
	\noindent
	\underline{Step 1:} First, we  construct ${\cal D}$ and ${\cal K}$  recursively. 
	Set  ${\cal D}_1=\{P_{\alpha/4}(x_1)\}$ and ${\cal K}_1=\{S_{\alpha/16}(x_1) \}$. %and ${\cal B}_1=\{{S_{\alpha/16}(x_1)}\}$.
	For each $y\in P_{\alpha/2}(x_2)$, the set $\varphi_{[-2\alpha,2\alpha]}(y)\cap P_{\alpha/2}(x_1)$
	is either one single point or empty, due to the fact that $P_{\alpha/2}(x_1)$ is a local cross section of time at least $2\alpha$. This yields that there is $t_y\in (-2\alpha,2\alpha)$ 
	such that $\varphi_{t_y}(y)\notin P_{\alpha}(x_1)$. Since  $X\setminus P_{\alpha/2}(x_1)$  is an open,
	using the continuity of the flow $\varphi: \R\times X\to X$, there are 
	an open interval $I_y\subset (-2\alpha,2\alpha)$ and an open neighbourhood
	$V_y\subset P_\alpha(x_2)$ of $y$ so that $\varphi_{I_y}(V_y)\subset 
	X\setminus P_{\alpha/2}(x_1)$, or equivalently,  $\varphi_{I_y}(V_y)\cap P_{\alpha/2}(x_1)=\varnothing$.
	Take $0<r_y<\alpha/4$ so small  that $P_{r_y}(y)\subset V_y$ to have
	\begin{equation*}\label{dk1} \varphi_{I_y} (P_{r_y}(y))\cap P_{\alpha/2}(x_1)=\varnothing.
	\end{equation*} 
	Due to the fact that $P_{\alpha/2}(x_2)$ is compact, there are $y_1,\dots,y_{n_2}\in P_{\alpha/2}(x_2)$ distinct such that
	$P_{r_{y_i}}(y_i)\subset P_\alpha(x_2)$ and
	\begin{equation*}\label{dk2}P_{\alpha/2}(x_2)\subset \bigcup_{i=1}^{n_2} \int S_{r_{y_i}/8}(y_i).
	\end{equation*} 
	Pick distinct numbers $u_1\in I_{y_1},\dots, u_{n_2}\in I_{y_{n_2}}$
	and set 
	\begin{align*}
		{\cal D}_2&={\cal D}_1\cup \{\varphi_{u_{1}}(P_{r_{y_1}}(y_{1})),\dots,\varphi_{u_{n_2}}(P_{r_{y_{n_2}}}(y_{n_2})) \},
		\\
		{\cal K}_2&={\cal K}_1\cup \{\varphi_{u_1}(S_{{r_{y_1}}/8}(y_{1})),\dots,\varphi_{u_{n_2}}(S_{r_{y_{n_2}}/8}(y_{n_2})) \}. 
		%{\cal B}_2&={\cal B}_1\cup \{\varphi_{u_1}(S_{{r_{y_1}}/3}(y_{1})),\dots,\varphi_{u_{n_2}}(S_{r_{y_{n_2}}/3}(y_{n_2})) \}.
	\end{align*} 
	Owing to that $u_1,\dots, u_{n_2}$ are distinct, we see that Poincar\'e sections in ${\cal D}_2$ are pairwise disjoint  satisfy Condition (c).
	Suppose that ${\cal D}_3,\dots,{\cal D}_{k-1}$, ${\cal K}_3,\dots,{\cal K}_{k-1}$ %  and ${\cal B}_3,\dots,{\cal B}_{k-1}$
	are similarly constructed
	for $k\leq m$ and all Poincar\'e sections in ${\cal D}_{k-1}$ satisfy Condition (c). We are going to construct ${\cal D}_k$ and ${\cal K}_k$. %, and ${\cal B}_k$. 
	Analogously to the construction of ${\cal D}_2$, for every
	$z\in P_{\alpha/2}(x_k)$, the set
	\[\varphi_{[-2\alpha,2\alpha]}(z)\cap {\tt D}_{k-1}\]
	is a set of finite points since ${\cal D}_{k-1}$ consists of finitely many local cross sections
	of times at least $2\alpha$; here  ${\tt D}_{k-1}$ denotes the union of elements in ${\cal D}_{k-1}$. Using the continuity of the flow, there exist an open interval $I_z\subset (-2\alpha,2\alpha)$
	and $0<r_z<\alpha/4$ such that $\varphi_{I_z}(P_{r_z}(z))\cap {\tt D}_{k-1}=\varnothing$.
	We cover $P_{\alpha/2}(x_k)$ by smaller rectangles $S_{r_{z_i}/8}(z_i)\subset P_{r_{z_i}}(z_i)\subset  P_{\alpha}(x_k)$:
	\begin{equation}\label{Rxk}
		P_{\alpha/2}(x_k)\subset \bigcup_{i=1}^{n_k} \int S_{r_{z_i}/8}(z_i),
	\end{equation} 
	where $z_i\in P_{\alpha/2}(x_k)$. 
	Pick distinct numbers $u_{1}\in I_{z_1},\dots, u_{n_k}\in I_{z_{n_k}}$
	and let  
	\begin{align*}
		{\cal D}_k&={\cal D}_{k-1}\cup\{\varphi_{v_1}(P_{r_{z_1}}(z_{1})),\dots,\varphi_{v_{n_k}}(P_{r_{z_{n_k}}}(z_{n_k})) \},
		\\
		{\cal K}_k&={\cal K}_{k-1}\cup \{\varphi_{v_1}(S_{r_{z_1}/8}(z_{1})),\dots,\varphi_{v_{n_k}}(S_{r_{z_{n_k}}/8}(z_{n_k})) \}.
		%\\
		%{\cal B}_k&={\cal B}_{k-1}\cup \{\varphi_{v_1}(S_{r_{z_1}/3}(z_{1})),\dots,\varphi_{v_{n_k}}(S_{r_{z_{n_k}}/3}(z_{n_k})) \}.
	\end{align*} 
	Due to the radii of elements in ${\cal D}_k$ is at most $\alpha/4$,
	their radii are at most $\alpha$ and hence ${\cal D}_k$ satisfies Condition (b).  Next we check that the elements in ${\cal D}_k$ satisfy Condition (c). 
	Suppose that $\varphi_{[-2\alpha,2\alpha]}(P_i)\cap P\ne \varnothing$ with $P\in {\cal D}_{k-1}$ and $P_i=\varphi_{v_i}(P_{r_{z_i}}(z_i))$ for some $i$. If $v_i\geq 0$, then $\varphi_{[0,2\alpha]}(P)\cap P_i=\varnothing$ and
	if  $v_i< 0$ then $\varphi_{[0,2\alpha]}(P_i)\cap P=\varnothing$.
	Let $P_i=\varphi_{v_i}(P_{r_{z_i}}(z_{i})),P_j=\varphi_{v_j}(P_{r_{z_j}}(z_j)),i\ne j$.
	If  $v_i>v_j$ then we observe that $\varphi_{[0,2\alpha]}(P_i)\cap P_j=\varnothing$.
	For, suppose on the contrary that there is $w=\varphi_t(u)\in P_j$ for $t\in [0,2\alpha]$ and $u\in P_i$.
	Then $w=\varphi_{v_j}(w')$ and $u=\varphi_{v_i}(u')$ for $u'\in P_{r_{z_i}}(z_i)\subset P_\alpha(x_k),w'\in P_{r_{z_j}}(z_j) \subset P_\alpha(x_k)$
	imply that $w=\varphi_{t+v_i}(u')=\varphi_{v_j}(w')$. 
	Since $P_\alpha(x_k)$ is a local cross section, we have $u'=w'$ and hence $t+v_i=v_j$
	or $t=v_j-v_i<0$, contradicting  $t\geq 0$. Similarly, if $v_i<v_j$, 
	then $P_i\cap \varphi_{[0,2\alpha]}(P_j)=\varnothing$. 
	We have shown that if $P, Q\in {\cal D}_{k}$ and $P\ne Q$, then at least one of the sets
	$\varphi_{[0,2\alpha]}(P)\cap Q$ and $\varphi_{[0,2\alpha]}(Q)\cap P$ is empty. Therefore, ${\cal D}_k$ satisfies Condition (c). 
	
	Repeating  this process, we obtain 
	\begin{eqnarray*}{\cal D}_m&=&{\cal D}_{m-1}\cup \{\varphi_{p_1}(P_{r_{w_1}}(w_{1})),\dots,\varphi_{p_{n_m}}(P_{r_{w_{n_m}}}(w_{n_m}))\},
		\\
		{\cal K}_m&=&{\cal K}_{m-1}\cup\{\varphi_{p_{1}}(S_{r_{w_1}/8}(w_{1})),\dots,\varphi_{p_{n_m}}(S_{r_{w_{n_m}}/8}(w_{n_m})) \},
		%\\
		%	{\cal B}_m&=&{\cal B}_{m-1}\cup\{\varphi_{p_{1}}(S_{r_{w_1}/3}(w_{1})),\dots,\varphi_{p_{n_m}}(S_{r_{w_{n_m}}/3}(w_{n_m})) \},
	\end{eqnarray*} 
	where $p_1\in I_{w_1},\dots, p_{n_m}\in I_{w_{n_m}}$ are pairwise distinct, $I_{w_1},\dots, I_{w_{n_m}}\subset (-2\alpha,2\alpha)$
	and $0<r_{w_i}<\alpha/4$ such that 
	$$P_{\alpha/2}(x_m)\subset \bigcup_{i=1}^{n_m}\int S_{r_{w_i}/8}(w_i)\quad\mbox{and}\quad \varphi_{I_{w_j}}(P_{r_{w_j}}(w_j))\cap {\tt D}_{m-1}=\varnothing,$$
	where ${\tt D}_{m-1}$ denotes the union of sets in ${\cal D}_{m-1}$, 
	$S_{r_{w_i}/8}(w_i)\subset P_{r_{w_i}}(w_i)\subset P_{\alpha}(x_k)$. 
	
	Let $n={\rm card\,}{\cal D}_m$ and denote the elements in ${\cal D}_m$ and ${\cal K}_m$ 
	by $D_1,\dots, D_n$, and
	$K_1,\dots,K_n$, respectively.  
	In summary, we have constructed a family of cross sections $D_1\dots, D_n$ satisfying conditions (b) and (c). 
	In addition, due to $\varphi_t(\Gamma g c_ub_s)=\Gamma g a_{t}c_{ue^t}b_{se^{-t}}$ for $g\in\PSL(2,\R), t,u,s\in\R$, 
	by correcting the radii of Poincar\'e sections $D_i$ and rectangles $K_i$, we may assume that
	\[ D_i=P_{4\eps}(z_i)\quad\mbox{and}\quad K_i=S_{\eps/2}(z_i)\]
	for $z_i\in X$ and some $\eps\in (0,\alpha/16)$.
	Then $D_i=P_{4\eps}(z_i)\subset P_{\alpha/4}(z_i)$, so  (b) holds by Lemma \ref{lcexpl}. 
	Now, for each $i\in\{ 1,\dots, n\}$,  define \[ B_i=S_{\eps}(z_i).\]
	to obtain (a). 
	
	\medskip 
	\noindent \underline{Step 2:} Proof of (d). 
	Due to \eqref{epsX}, for any $x\in X$,  either $x\in \varphi_{[-\alpha,0]}(S_{\alpha/16}(x_1))$ or 
	$x\in \varphi_{[-\alpha,0]}(\int P_{\alpha/2}(x_k))$ for some  $k\in\{2,\dots,m\}$.
	Then \eqref{Rxk} implies that   $x\in \varphi_{[-\alpha,0]}(\int S_{r_{z_i}/8}(z_i))$
	for some $i\in \{ 1,\dots, r_{n_k}\}$.  This means that $x\in \varphi_{[-\alpha,0]}(\int K_s)$ for some $s\in\{1,\dots,n\}$, and the former of (d) is proved. This yields the latter of (d). 
	
	\medskip 
	\noindent 
	\underline{Step 3:} Proof of (e). Write $z_i=\Gamma g_i$ for $g_i\in \PSL(2,\R)$.
	Suppose that $x=\varphi_t(y)$ for $t\in [-\alpha,\alpha]$, $x\in B_i$ and $y\in B_j$. 
	We need to check  that $x\in \varphi_{[-2\alpha,2\alpha]}(D_j)$.
	Recall that for $k\in \{1,\dots,n\}$, 
	\[B_k=S_{\eps}({z_k}) =\{\Gamma g_k c_ub_s, u\in [-\eps, \eps], s=\frac{s'}{1-us'}\ \mbox{for some} \ s'\in [-\eps, \eps] \}.\]
	We have $x=\Gamma g_i c_ub_s=\Gamma g_j  c_{\hat u}b_{\hat s}a_t$. %for some $u,u'\in [-2\eps,2\eps ]$,
	For any $z=\Gamma g_i c_{\tilde u}b_{\tilde s}\in B_i$,  we write
	\begin{align*}
		z&\ =\Gamma g_i  c_ub_s b_{-s}c_{-u} c_{\tilde u}b_{\tilde s}
		=\Gamma g_j  c_{\hat u}b_{\hat s}a_t b_{-s}c_{-u} c_{\tilde u} b_{\tilde s}
		\\
		&\ =\Gamma g_j c_{\hat u}b_{\hat s-se^{t}}c_{(\tilde u-u)e^{-t}}b_{\tilde s e^t}a_t
		=\Gamma g_jc_{\bar u}b_{\bar s}a_{\bar t},
	\end{align*}
	where
	\begin{align*}
		\bar s &\ = \hat s-se^{t}+\tilde s e^t+(\tilde u-u)e^{-t}(\hat s-se^{t})(1+\hat s-se^{t}+\tilde s e^t)
		\\ &\ \ \ \ \  +
		(\tilde u-u)(\hat s-se^{t})\tilde s (1+(\tilde u-u)e^{-t}(\hat s-se^{t})),
		\\
		\bar u&\ = \hat u+(\tilde u-u)e^{-t}-\frac{(\tilde u-u)^2e^{-2t}({\hat s-se^{t}})}{1+(\tilde u-u)e^{-t}({\hat s-se^{t}})},
		\\
		\bar t&\ = t+2\ln(1+(\tilde u-u)e^{-t}({\hat s-se^{t}} )).
	\end{align*}
	After a short calculation, we obtain $|\bar t|\leq 2\alpha, |\bar s|<4\eps, |\bar u|< 4\eps$.
	This means that $z\in \varphi_{[-2\alpha,2\alpha]}(D_j)$, proving Condition (e). 
	
	The theorem is proved.
\end{proof}

From the above proof, it follows the next result.
\begin{remark}\label{e'}\rm With the setting in Theorem \ref{step1},
	
	(e') if $B_i\bigcap \varphi_{[-\alpha,0]}(B_j)\ne\varnothing$
	then $B_i\subset \varphi_{[-2\alpha,0]}(D_j)$ and if $B_i\bigcap \varphi_{[0,\alpha]}(B_j)\ne\varnothing$
	then $B_i\subset \varphi_{[0,2\alpha]}(D_j)$. 
	{\hfill$\Diamond$}
\end{remark}

Now, recall that $X=\varphi_{[-\alpha,0]}(\bigcup_{i=1}^n K_i)=\varphi_{[-\alpha/2,\alpha/2]}(\bigcup_{i=1}^n K_i)$. Let $3\lambda>0$ be the Lebesgue number for the cover $\{\varphi_{[-\alpha/2,\alpha/2]}(K_1),\dots, \varphi_{[-\alpha/2,\alpha/2]}(K_n)\}$, i.e.,
any subset of $X$ with diameter at most  $3\lambda$ contains in some $\varphi_{[-\alpha/2,\alpha/2]}(K_i)$.
Fix $L>0$ and $i\in\{1,\dots,n\}$. For $x\in K_i$, there is a closed neighbourhood $V_x^i$ of $x$  such that 
\[\diam\varphi_t(V_x^i)\leq \lambda \quad \mbox{for all}\quad |t|\leq 2L.\]
Since $K_i$ compact, we cover it by a finite family ${\cal V}_i=\{V^i_{x_1},\dots, V^i_{x_{n_i}}\}$:
\[K_i\subset\bigcup_{j=1}^{n_i} V^i_{x_j}.\] 
We may assume that for any $V\in {\cal V}_i$, $V\subset S_{2\eps/3}(z_i)$.

Let $A\subset X$ be given. Denote 
\[B(A,\lambda)=\{x\in X: d_X(x,A)=\inf_{y\in A}d_X(x,y)<\lambda\}.\]
We claim that $\diam B(\varphi_{-L}(V),\lambda)<3\lambda$
and $\diam B(\varphi_L(V),\lambda)<3\lambda$ for all $V\in {\cal V}_i$. 
For, taking $x,y\in B(\varphi_{-L}(V),\lambda)$, 
there are $z_1,z_2\in \varphi_{-L}(V)$ such that 
$d_X(x,z_1)<\lambda$ and $d_X(y,z_2)<\lambda$. This implies
$d_X(x,y)\leq d_X(x,z_1)+d_X(z_1,z_2)+d_X(z_2,y)<3\lambda$
and hence $\diam B(\varphi_{-L}(V),\lambda)<3\lambda$. Similarly, $\diam B(\varphi_{L}(V),\lambda)<3\lambda$.

By the property of $\lambda$, for each $i\in\{1,\dots,n\}$ and $V\in {\cal V}_i$,
there  are $a(V),b(V)\in\{1,\dots,n\}$ so that 
$B(\varphi_{-L}(V),\lambda)\subset \varphi_{[-\alpha/2,\alpha/2]}(K_{a(V)})$
and $B(\varphi_{L}(V),\lambda)\subset \varphi_{[-\alpha/2,\alpha/2]}(K_{b(V)})$. Then the following maps
\[g_{V^-}=\pr_{D_{a(V)}}\circ\varphi_{-L}:V\longrightarrow K_{a(V)}\subset B_{a(V)}\]
and
\[g_{V^+}=\pr_{D_{b(V)}}\circ\varphi_{L}:V\longrightarrow K_{b(V)}\subset B_{b(V)}\]
are well-defined.
We recursively define the set $R_{i,k}$ and $S_{i,k}$ by
$R_{i,0}=S_{i,0}=K_i$ and for $k\geq 0$
\begin{align} 
	R_{i,k+1}&=\bigcup\limits_{V\in {\cal V}_i}\bigcup\limits_{v\in V}\{\langle y, \pr_{D_i}\varphi_L(z)\rangle_{D_i}:
	y\in K_i, z\in W^s(g_{V^-}(v),R_{a(V),k})  \}, \\
	S_{i,k+1}&=\bigcup\limits_{V\in {\cal V}_i}\bigcup\limits_{v\in V}\{\langle \pr_{D_i}\varphi_{-L}(z),y\rangle_{D_i}:
	y\in K_i, z\in W^u(g_{V^+}(v),S_{b(V),k}) \}. \label{Sik}
\end{align}
%Notice that for $x\in V$, then $g_{V^-}(x)\in K_{a(V)}$ and 
%$R_{a(V),k}\subset K_{a(V)}, W^s(g_{V^-}(x),R_{a(V),k})$ makes sense  and also $W^s(g_{V^-}(x),R_{a(V),k})\subset K_{a(V)}$. 
%
%
%Note that $B_\lambda (\varphi_{-L}(V))\subset \varphi_{[-2\alpha,2\alpha]}K_{a(V)}$, then 
%$\varphi_{L}(K_{a(V)})\subset \varphi_{[-\theta,\theta]}(K_i)$ and hence 
%$\langle y, P_{D_i}\varphi_L(z)\rangle_{D_i}$ makes sense. 

For $x\in (0,1)$, we set $x'=\frac{x}{1-x^2}$. Note that $x<y$ if and only if $x'<y'$ and $x=y$ if and only if $x'=y'$.

In the rest of the paper, we consider $L>4$ and $T:=L-\alpha/2$. 
Define $\eps_0=2\eps/3$ and
$\eps_{k+1}=\eps_0 +2\eps_k e^{-T}, k\geq 0$.
Accordingly, $\eps_0'=2\eps'/3$ and $\eps_{k+1}'=\eps_0'+2\eps_k' e^{-T}$, $k\geq 0$. 
\begin{lemma} \label{Riklm} For every $i\in \{1,\dots,n\}$, the following statements hold.
	
	(a) $R_{i,k}\subset S_{\eps_0}^{\eps_{k}}(z_i)$ and 	
	$S_{i,k}\subset S^{\eps_k}_{\eps_0}(z_i)$ for all $k\geq 0$.  
	
	(b) The sets $R_i=\bigcup\limits_{k=0}^\infty R_{i,k},\, S_i=\bigcup\limits_{k=0}^\infty S_{i,k}$
	are subsets of $B_i$.
	
	(c) The set $C_i=\langle S_i,R_i\rangle_{D_i}=\{\langle p,q\rangle_{D_i}:p\in S_i, q\in R_i  \}$
	is a rectangle contained in $B_i$. 
\end{lemma}

\begin{proof}   (a) We prove the former by induction. 
	First,	
	$$R_{i,1}=\underset{V\in {\cal V}_i}\bigcup\bigcup_{v\in V}\{\langle y, \pr_{D_i}\varphi_L(z)\rangle_{D_i}:
	y\in K_i, z\in W^s(g_{V^-}(v), K_{a(V)})\}.$$
	For any $x\in R_{i,1}$,  $x=\langle y, \pr_{D_i}\varphi_{L}(z)\rangle_{D_i}$
	for $z\in W^s(\pr_{D_{a(V)}}\varphi_{-L}(v),K_{a(V)})$
	with some $v\in V\in {\cal V}_i$. We first show that 
	$\pr_{D_i}\varphi_{L}(z)\in S_{\eps_0}^{\eps_1}(z_i)$.
	Let $v=\Gamma g_i c_{u_v}b_{s_v}\in V\subset S_{2\eps/3}(z_i)$ and $\pr_{D_{a(V)}}\varphi_{-L}(v)=\varphi_{-L-\tau}(v)=\Gamma g_l c_{u}b_s\in K_l=S_{\eps/2}(\Gamma g_l)$ with $l={a(V)}\in\{1,\dots,n\}$ for some $\tau\in [-\alpha/2,\alpha/2]$.
	According to Proposition \ref{WSlm}\,(a), $z=\Gamma g_l c_{u}b_{s_z}$.
	Since $\varphi_{L+\tau}(\pr_{D_{a(V)}}\varphi_{-L}(v))=v$,  
	it follows that $\Gamma g_l c_ub_s a_{L+\tau}=\Gamma g_i c_{u_v}b_{s_v}$.
	This implies that $$\hat z:=\varphi_{L+\tau}(z)=\Gamma g_l c_ub_sa_{L+\tau}b_{(s_z-s)e^{-L-\tau}}
	=\Gamma g_i c_{u_v}b_{s_v+(s_z-s)e^{-L-\tau}}\in D_i.$$
	Also $\hat z=\pr_{D_i}\varphi_{L}(z)=(u_{\hat z}, s_{\hat z})_{z_i}\in D_i$,
	where $$u_{\hat z}=u_v\ \mbox{ and }\ s_{\hat z}=(u,s_v+(s_z-s)e^{-L-\tau})_{z_i}.$$ 
	Then $$|s_{\hat z}|\leq \eps_0' +2\eps_0' e^{-L+\alpha/2}\leq\eps_0' +2\eps_0' e^{-T}=\eps_1'$$ shows that 
	$\hat z\in S^{\eps_{1}}_{\eps_0}(z_i)$. Since $y\in K_i=S_\eps(z_i)\subset S_{\eps_0}^{\eps_1}(z_i)$,
	we get $x=\langle y,\hat z\rangle_{D_i}\in S_{\eps_0}^{\eps_{1}}(z_i)$ due to the fact that $S_{\eps_0}^{\eps_{1}}(z_i)$ 
	is a rectangle. Therefore $R_{i,1}\subset S_{\eps_0}^{\eps_1}(z_i)$.
	
	Next, assume that $R_{i,j-1}\in S^{\eps_{j-1}}_{\eps_0}(z_i)$ for $j>1$.
	%We show that $R_{i,k}\in S^{\eps_{k}}_\eps(z_i)$.
	%$$R_{i,k+1}=\underset{V\in {\cal V}_i}\bigcup\{\langle y, P_{D_i}\varphi_L(z)\rangle_{D_i}:	y\in K_i, z\in W^s(g_{V^-}(v), R_{a(V),k})\mbox{ with } x\in V \}.$$
	Take $x=\langle y, \pr_{D_i}\varphi_{L}(z)\rangle_{D_i}\in R_{i,j}$, where $y\in K_i$, $z\in W^s(\pr_{D_{a(V)}}\varphi_{-L}(v), R_{a(V),j-1})$ for some $v=\Gamma g_i c_{u_v}b_{s_v}\in V\in {\cal V}_i$.
	Similarly to above, 
	$\pr_{D_{a(V)}}\varphi_{-L}(v)=\varphi_{-L-\tau}(v)=\Gamma g_l c_{u}b_s\in R_{l,j-1}$ with $l={a(V)}$ for some $\tau\in [-\alpha/2,\alpha/2]$.
	Writing $z=\Gamma g_l c_{u}b_{s_z}$, we have 
	\begin{align*} 
		\hat z:=&\ \pr_{D_i}\varphi_{-L}(z)=\varphi_{L+\tau}(z)=\Gamma g_l c_ub_sa_{L+\tau}b_{(s_z-s)e^{-L-\tau}}\\
		=&\ \Gamma g_i c_{u_v}b_{s_v+(s_z-s)e^{-L-\tau}}
		=(u_{\hat z}, s_{\hat z})_{z_i}\in D_i,
	\end{align*} 
	where $$u_{\hat z}=u_v,\ s_{\hat z}=s_v+(s_z-s)e^{-L-\tau}.$$ Then 
	%Furthermore, $\hat z\in W^s(v, B_i)$ since $u_{\hat z}=u_v$ and 
	%$s_{\hat z}=s_v+(s_z-s)e^{-L-\tau}$.
	$$|s_{\hat z}|\leq \eps'_0 +2\eps'_{j-1} e^{-T}=\eps'_{j}$$
	yields
	$\hat z\in S^{\eps_{j}}_{\eps_0}(z_i)$. Since $y\in K_{i}\subset S_{\eps_0}^{\eps_j}(z_i)$, 
	we obtain  $x=\langle y,\hat z\rangle_{D_i} \in S_{\eps_0}^{\eps_{j}}(z_i)$ and so $R_{i,j}\subset S_{\eps_0}^{\eps_j}(z_i)$. We deduce that $R_{i,k}\subset S_{\eps_0}^{\eps_k}(z_i)$ for all $k\geq 0$. 
	
	To verify the latter, we need the other versions of Poincar\'e sections and rectangles. 
	Define
	$$\widetilde D_i=\widetilde P_{4\eps }(z_i)=\{\Gamma g_i b_s c_u: u,s\in [-4\eps,4\eps]\}$$
	and $$\widetilde B_i=T_{\eps}(z_i)=\{\Gamma g_i b_s c_u: s\in [-\eps,\eps]\mbox{ and } u=\frac{u'}{1-su'}\quad \mbox { for some } u'\in [-\eps,\eps]\}.$$ 
	We recall from Lemma \ref{rec-pro-lm} that $\pr_{D_i}(\widetilde B_i)=B_i$ and $\pr_{\widetilde D_i}(B_i)=\widetilde B_i$. 
	By projecting to $\widetilde D_i$,
	\eqref{Sik} is equivalent to
	\begin{equation}\label{tildeS}
		\widetilde S_{i,k+1}
		=\bigcup\limits_{\widetilde V\in \widetilde {\cal V}_i}\bigcup\limits_{v\in \widetilde V}\{\langle \pr_{\widetilde D_i}\varphi_{-L}(z),y\rangle_{\widetilde D_i}:
		y\in \widetilde K_i, z\in W^u(g_{\widetilde V^+}(v),S_{b(\widetilde V),k}) \},
	\end{equation}
	where $\widetilde S_{i,k}=\pr_{\widetilde D_i}(S_{i,k}),$
	$\widetilde K_i=\pr_{\widetilde D_i}(K_i)$, $\widetilde{\cal V}_i=\pr_{\widetilde D_i}({\cal V}_i)$,
	$\widetilde V_i=\pr_{\widetilde D_i}(V_i)$, and
	\[g_{\widetilde V^+}=\pr_{\widetilde D_{b(\widetilde V)}}\circ\varphi_L:\widetilde V\longrightarrow \widetilde K_{b(\widetilde V)}\subset \widetilde B_{b(\widetilde V)}.\]
	Recall from Proposition \ref{rec-pro-lm} that $\pr_{\widetilde D_i}(S_{\eps_0}^{\eps_k}(z_i))=T^{\eps_0}_{\eps_k}(z_i)$
	and $\pr_{D_i}(T^{\eps_0}_{\eps_k}(z_i))=S_{\eps_0}^{\eps_k}(z_i)$. Together with  $\pr_{D_i}(\widetilde S_{i,k})=S_{i,k}$, the inclusion $S_{i,k}\subset S_{\eps_k}^{\eps_0}(z_i)$ is equivalent to
	\begin{equation}\label{eqll}\widetilde S_{i,k}\subset T_{\eps_0}^{\eps_k}(z_i).
	\end{equation}
	%Take $x=\langle y, \pr_{\widetilde D_i}\varphi_{-L}(z)\rangle_{\widetilde D_i}\in \widetilde S_{i,1}$
	We first verify that $\widetilde S_{i,1}\subset T_{\eps_0}^{\eps_1}(z_i)$. 
	For	$x\in \widetilde S_{i,1}$, $x=\langle \pr_{\widetilde D_i}\varphi_{-L}(z),y\rangle_{\widetilde D_i}$,
	where  $y\in \widetilde K_i$ and $z\in W^u( \pr_{\widetilde D_{b(\widetilde V)}}\varphi_{L}(v),\widetilde K_{b(\widetilde V)})$
	for $v=\Gamma g_i b_{s_v}c_{u_v}\in \widetilde V\in \widetilde{\cal V}_i$.
	There is a $\tau\in [-\alpha/2,\alpha/2]$ such that 	  $\pr_{\widetilde D_{b(\widetilde V)}}\varphi_{L}(v)=\varphi_{L+\tau}(v)=\Gamma g_l b_sc_{u}\in \tilde K_l=T_{\eps/2}(\Gamma g_l)$ for $l={b(\widetilde V)}$.
	By Proposition \ref{WSlm}\,(b), $z=\Gamma g_l b_{s}c_{u_z}$.
	Since $\varphi_{-L-\tau}(\pr_{\widetilde D_{b(\widetilde V)}}\varphi_{L}(v))=v$,  it follows that $\Gamma g_l b_sc_u a_{-L-\tau}=\Gamma g_i b_{s_v}c_{u_v}$.
	Then $$\hat z:=\varphi_{-L-\tau}(z)=\Gamma g_l b_sc_ua_{L+\tau}c_{(u_z-u)e^{-L-\tau}}
	=\Gamma g_i b_{s}c_{u_v+(u_z-u)e^{-L-\tau}}\in \widetilde D_i$$
	implies that
	$$\hat z=\pr_{\widetilde D_i}\varphi_{-L}(z)=(s_{\hat z}, u_{\hat z})_{z_i}',$$
	where \[ s_{\hat z}=s, \ u_{\hat z}=u_v+(u_z-u)e^{-L-\tau}.\] 
	The estimate $$|u_{\hat z}|\leq \eps_0' +2\eps_0' e^{-L+\alpha/2}\leq \eps_0' +2\eps_0' e^{-T}=\eps_1'$$ shows that 
	$\hat z\in T^{\eps_0}_{\eps_1}(z_i)$. Since $y\in \tilde K_i=T_{\eps/2}(z_i)\subset T_{\eps_1}^{\eps_0}(z_i)$, 
	we have  $x\in T^{\eps_0}_{\eps_{1}}(z_i)$ due to the fact that $T^{\eps_0}_{\eps_{1}}(z_i)$ 
	is a rectangle, and we deduce $\widetilde S_{i,1}\subset T^{\eps_0}_{\eps_1}(z_i)$.
	In the similar way, we can show that if $\widetilde S_{i,j-1}\subset T^{\eps_0}_{\eps_{j-1}} (z_i)$ for $i>1$, then 
	$\widetilde S_{i,j}(z_i)\subset T^{\eps_0}_{\eps_{j}}(z_i)$ and hence \eqref{tildeS} is obtained. 
	
	(b) We have
	\begin{align*}
		\eps_k&=\eps_0+\eps_0 (2e^{-T})+ \dots + \eps_0 (2e^{-T})^{k-1}+ \eps_0 (2e^{-T})^k
		\\
		&=\frac{1-(2e^{-T})^{k+1}}{1-2 e^{-T}}\eps_0
		<\frac{\eps_0}{1-2e^{-T}}<\frac{3}{2}\eps_0=\eps.
	\end{align*} for all $k\geq 1$ when $T>3$. This implies that 
	$S_{\eps_0}^{\eps_k}(z_i)\subset S_{\eps}(z_i)= B_i$ and $S^{\eps_0}_{\eps_k}(z_i)\subset S_{\eps}(z_i)= B_i$
	for all $k\geq 1$. Therefore $R_i=\bigcup\limits_{k=0}^\infty R_{i,k}\subset B_i$ and
	$S_i=\bigcup\limits_{k=0}^\infty S_{i,k}\subset B_i$.

	(c) If $x_i=\langle p_i,q_i\rangle_{D_{i}}\in C_i, i=1,2$,  then
	$\langle x_1,x_2\rangle_{D_i}=
	\langle \langle p_1,q_1\rangle_{D_{i}},
	\langle p_2,q_2\rangle_{D_{i}}
	\rangle_{D_i}
	=\langle p_1,q_2\rangle_{D_i}\in C_i$ by Lemma \ref{=x=ylm}\,(d).
	This shows that each $C_i$ is a rectangle. Due to (b), it follows that
	$C_i\subset B_i$.
\end{proof}

\begin{figure}[h!]
	\begin{center}
		\begin{minipage}{1\linewidth}
			\centering
			\includegraphics[angle=0,width=1\linewidth]{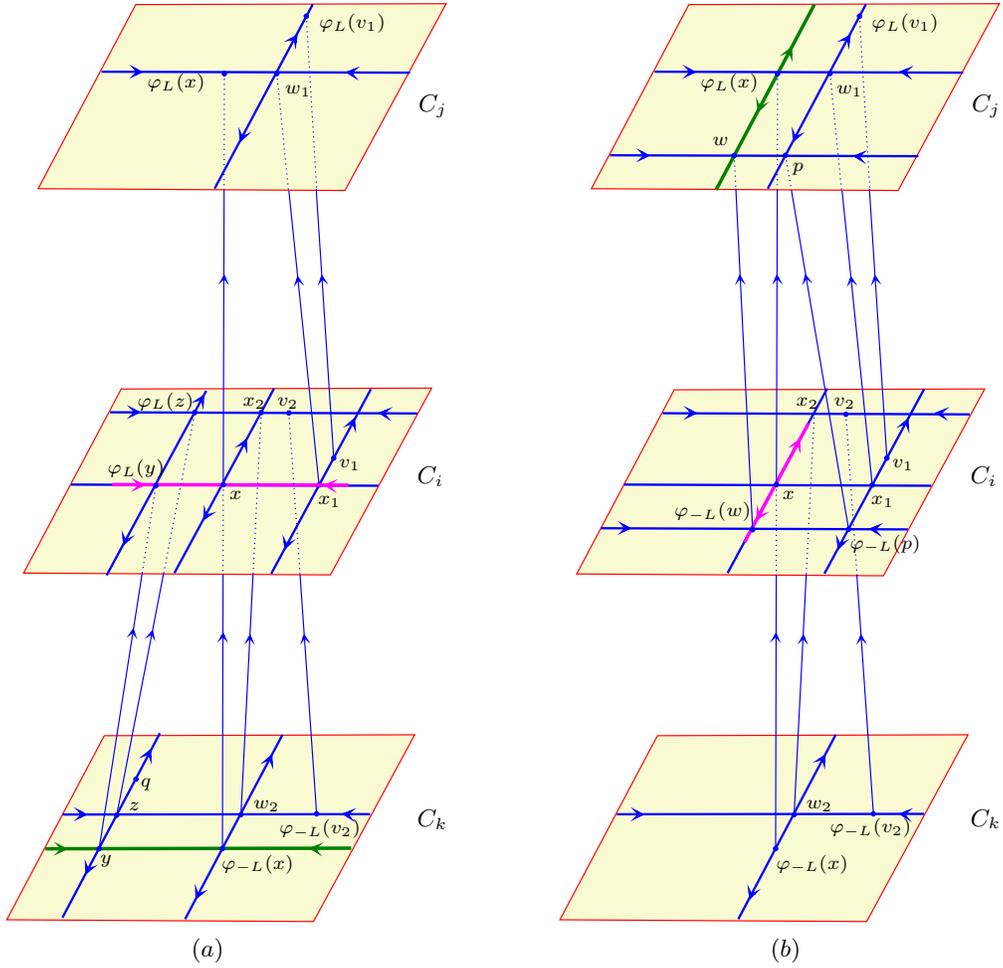}
		\end{minipage}
	\end{center}
	\caption{Illustration for the proof of Lemma \ref{CiCj} }\label{RS}
\end{figure}
The next lemma is a key result, which help us prove the final statement (Lemma \ref{flm2}). 
\begin{lemma}\label{CiCj}
	Consider $x\in C_i$. 
	
	\noindent 
	(a) There is a $k$ so that 
	\[\varphi_{-L}(x)\in \varphi_{[-\alpha/2,\alpha/2]}(C_k)
	\quad\mbox{and}\quad 
	\pr_{D_i}\varphi_{L}W^s(\pr_{D_k}\varphi_{-L}(x),C_k)\subset W^s(x,C_i). 
	\]
	
	\noindent
	(b) There is a $j$ so that
	\[\varphi_{L}(x)\in \varphi_{[-\alpha/2,\alpha/2]}(C_i)
	\quad\mbox{and}\quad 
	\pr_{D_i}\varphi_{-L}W^u(\pr_{D_j}\varphi_L(x),C_i)\subset W^u(x,C_i).
	\] 
\end{lemma}
\begin{proof} (a)  Since $X=\cup_{i=1}^n \varphi_{[-\alpha/2,\alpha/2]}(K_i)$,
	it follows the former, also $\pr_{D_k}\varphi_{-L}(x)$ makes sense.
	Write $x=\langle x_1,x_2\rangle_{D_i}\in C_i$,
	where $x_1=\pr_{D_i}\varphi_{-L}(w_1), x_2=\pr_{D_i}\varphi_{L}(w_2)$
	for $w_1\in W^u(\pr_{D_j}\varphi_{L}(v_1), S_j)$, $w_2\in W^s(\pr_{D_k}\varphi_{-L} (v_2), R_k)$
	for some $v_1,v_2\in S_{2\eps/3}(z_i)$; see Figure \ref{RS}\,(a) for an illustration.
	Write $x=\Gamma g_i c_{u_{ x}}b_{s_{x}}\in C_i$
	and $\pr_{D_k}\varphi_{-L}(x)=\varphi_{-L-\tau}(x)=\Gamma g_i c_{u_{ x}}b_{s_{x}} a_{-L-\tau}=\Gamma g_k c_u b_s \in C_k$
	for some $\tau\in [-\alpha/2,\alpha/2]$.
	If $y\in W^s(\pr_{D_k}\varphi_{-L}(x), C_k)$, then by Proposition \ref{WusT}\,(a), $y=\Gamma g_k c_u b_{s_y}$.
	It follows that $$\varphi_{L+\tau}(y)=\Gamma g_k c_u b_{s}a_{L+\tau}b_{(s_y-s)e^{-L-\tau}}
	=\Gamma g_j c_{u_x} b_{s_x+(s_y-s)e^{-L-\tau} }\in W^s(x, D_i).$$
	This means that $\pr_{D_i}\varphi_L(y)=\varphi_{L+\tau}(y)\in W^s(x, D_i)$,
	and it remains to show $\varphi_{L+\tau}(y)\in C_i$. 
	For, let $z=\langle \pr_{D_k}\varphi_{-L}(v_2),y\rangle_{D_k}\in C_k$.
	We check that $\pr_{D_i}\varphi_L(z)\in C_i$ and also $\pr_{D_i}\varphi_L(y)=\langle x,\pr_{D_i}\varphi_L(z)\rangle_{D_i}$.
	
	Write $v_2=\Gamma g_i c_{u_{v_2}}b_{s_{v_2}}\in S_{2\eps/3}(z_i)$ and 
	$\pr_{D_k}\varphi_{-L}(v_2)=\varphi_{-L-r}(v_2)=\Gamma g_i c_{u_{v_2}}b_{s_{v_2}}a_{-L-r}= \Gamma g_k c_{\hat u}b_{\hat s}\in C_k$
	for some $r\in [-\alpha/2,\alpha/2]$. Then $z=\Gamma g_k c_{\hat u}b_{s_z}$ yields
	$$\varphi_{L+r}(z)=\Gamma g_i c_{u_{v_2 }}b_{s_{v_2}+(s_z-\hat s)e^{-L-\tau}}\in C_i$$
	by the construction of $C_i$; see the proof of Lemma \ref{Riklm}\,(a).
	Next, since $y\in W^s( \pr_{D_k}\varphi_{-L}(x), C_k)$, $y=\langle \pr_{D_k}\varphi_{-L}(x), q\rangle_{D_k}$ for some $q\in C_k$. It follows from Lemma \ref{=x=ylm}\,(a) that 
	\[z=\langle \pr_{D_k}\varphi_{-L}(v_2), y\rangle_{D_k}
	=\langle \pr_{D_k}\varphi_{-L}(v_2),\langle \pr_{D_k}\varphi_{-L}(x), q\rangle_{D_k}\rangle_{D_k}
	=\langle \pr_{D_k}\varphi_{-L}(v_2), q\rangle_{D_k}.\]
	This implies 
	\[y=\langle \pr_{D_k}\varphi_{-L}(x), q\rangle_{D_k}=\langle \pr_{D_k}\varphi_{-L}(x), \langle \pr_{D_k}\varphi_{-L}(v_2),q\rangle_{D_k}\rangle_{D_k}
	= \langle \pr_{D_k}\varphi_{-L}(x),z\rangle_{D_k}.\]
	We are in a position to show that
	$\pr_{D_i}\varphi_{L}(y)=\langle x, \pr_{D_i}\varphi_L(z)\rangle_{D_i}$. There is no loss of generality, we may assume that $r\leq \tau$. Define 
	\begin{equation}
		s(t)=\begin{cases}
			t&\mbox{if}\quad t\in [0, L+r],\\
			L+r&\mbox{if}\quad t\in [L+r,L+\tau].
		\end{cases}
	\end{equation}
	Then $s: [0,L+\tau]\to \R$ is continuous with $s(0)=0$ and for $t\in [0,L+\tau]$
	\begin{align}\label{eqC1}
		d_X(\varphi_t(\varphi_{-L-\tau}(x)),\varphi_{s(t)}(z))
		\leq d_X(\varphi_t(\varphi_{-L-\tau}(x)),\varphi_t(w_2))
		+d_X(\varphi_t(w_2),\varphi_{s(t)}(z)).
	\end{align}
	Owing to $z\in W^s( w_2, C_k)$ and $s(t)=t$, $t\in [0,L+r]$, it follows that
	for $t\in [0,L+r]$
	\[d_X(\varphi_t(w_2),\varphi_{s(t)}(z))<\eps e^{-t},\]
	whence for $t\in [L+r,L+\tau]$
	\begin{align*}
		d_X(\varphi_{t}(w_2),\varphi_{s(t)}(z))
		&= d_X(\varphi_t(w_2),\varphi_{L+r}(z))\\
		&\leq d_X(\varphi_{t}(w_2), \varphi_{L+r}(w_2))
		+d_X(\varphi_{L+r}(w_2),\varphi_{L+r}(z))\\
		&\leq \frac{1}{\sqrt 2}|t-(L+r)|+ \eps \leq  \frac{1}{\sqrt 2}|\tau-r|+\eps<2\alpha,
	\end{align*} 
	where we have used Lemma \ref{tslm}. Hence
	\begin{equation}\label{eqC2}d_X(\varphi_{t}(w_2),\varphi_{s(t)}(z))< 2\alpha\ \mbox{ for all } t\in [0, L+\tau].
	\end{equation}
	In addition, since $x\in W^u(x_2, C_i)$, $d_X(\varphi_t(x),\varphi_t(x_2))<\eps e^{t}$ for all $t\in\R$.
	Write  $\varphi_{-L-\bar\tau}(x_2)=w_2$ for some $\bar\tau\in [-\alpha/2,\alpha/2]$.
	Let us assume that $\tau\leq\bar\tau$. 
	For $t\in [0,L+\tau]$, 
	\begin{align}\notag 
		d_X(\varphi_t(\varphi_{-L-\tau}(x)),\varphi_t(w_2))
		&\leq d_X(\varphi_t(\varphi_{-L-\tau}(x)), \varphi_{t+\tau-\bar \tau}(\varphi_{-L-\tau}(x)))
		\\ \notag 
		&\ \ \ +d_X(\varphi_{t+\tau-\bar\tau}(\varphi_{-L-\tau}(x)),\varphi_t(\varphi_{-L-\bar\tau}(x_2)))
		\\ \notag 
		&<\frac{1}{\sqrt 2}|\tau-\bar\tau|+d_X(\varphi_{t-L-\bar\tau}(x),\varphi_{t-L-\bar\tau}(x_2))\\
		& \label{eqC3}
		<\alpha+\eps <2\alpha.
	\end{align}
	Combining \eqref{eqC1}-\eqref{eqC3}, we get
	\[d_X(\varphi_t(z),\varphi_{s(t)}(x))<4\alpha<\delta_2\mbox{\ for all \ } t\in [0,L+\tau].\]
	Apply Lemma \ref{DD'} to obtain
	$\langle x,\varphi_{L+r}(z) \rangle_{D_i} 
	=\pr_{D_i}\varphi_{L+\tau}(y)=\pr_{D_i}\varphi_{L}(y)$, which proves that $y\in W^s(x, C_i)$.

	(b)  The former is clear and so $\pr_{D_j}\varphi_L(x)$ makes sense.
	For $w\in W^u(\pr_{D_j}\varphi_L(x),C_j)$, we need to verify that $\pr_{D_i}\varphi_{-L}(w)\in W^u(x,C_i)$.
	We need the other version of rectangles. 
	The inclusion is equivalent to 
	$\pr_{\widetilde D_i}\varphi_{-L}W^u(\pr_{\widetilde D_j}\varphi_L(\tilde x),\widetilde C_j)\subset W^u(\tilde x,{\widetilde C}_i)$,
	where $\tilde x=\pr_{{\widetilde D}_i}(x), \widetilde C_i=\pr_{\widetilde D_i}(C_i), 
	\widetilde C_j=\pr_{\widetilde D_j}(C_j)$. 
	
	Write $\tilde x=\Gamma g_i b_{s_{\tilde x}}c_{u_{\tilde x}}\in {\widetilde C}_i$
	and $\varphi_{L+\nu}(\tilde x)=\pr_{\widetilde D_j}\varphi_L(\tilde x)=\Gamma g_j b_sc_u\in \widetilde C_j$
	and $\tilde w=\Gamma g_j b_s c_{u_{\tilde w}}\in W^u(\pr_{\widetilde D_j}\varphi_L(\tilde x),\widetilde C_j)$; see Proposition \ref{WSlm}\,(b).
	Then $\varphi_{-L-\nu}(\tilde w)=\Gamma g_ib_{s_{\tilde x}}c_{u_{\tilde x} +(u_{\tilde w}-u)e^{-L-\nu}}
	\in W^u(\tilde x,{\widetilde D}_i)$. Equivalently,  $\pr_{D_i}\varphi_{-L}(w)\in W^u(x,D_i)$.
	Let $p=\langle w, w_1\rangle_{D_j}$. Analogously to above, we can verify that
	$\pr_{D_i}\varphi_{-L}(p)\in C_i$ and 
	$\pr_{D_i}\varphi_{-L}(w)=\langle \pr_{D_i}\varphi_{-L}(p),x \rangle_{D_i}\in W^u( x, C_i )$; see
	Figure \ref{RS} for a depiction. The proof is complete.
\end{proof}
\begin{proposition} If $L>\ln(4\eps/\lambda)$, then $\varphi_{-L}W^u(x,C_i)\subset \varphi_{[-\alpha/2,\alpha/2]}(C_k)$ 
	and $\varphi_{L}W^s(x,C_i)\subset \varphi_{[-\alpha/2,\alpha/2]}(C_j)$.
\end{proposition}
\begin{proof}
	Write $x=\Gamma g\in C_i$ with $g\in\PSL(2,\R)$ and fix $L>\ln(4\eps/\lambda)$ or $e^{-L}<10\eps/{\lambda}$. 
	For any $y\in W^u(x,C_i)$, $y=\langle z,x\rangle_{C_i}$, so  $y=\Gamma gc_ua_\tau$ for some $u,\tau\in[-3\eps,3\eps]$;
	see Lemma \ref{proexa}.
	Then 
	\begin{align*}
		d_X(\varphi_{-L}(y),\varphi_{-L+\tau}(x))
		&=d_X(\Gamma gc_u a_{-L+\tau},\Gamma g a_{-L+\tau})
		\leq d_\G(c_ua_{-L+\tau},a_{-L+\tau} )
		\\
		&=d_\G(c_{ue^{-L+\tau}},e)
		\leq |u|e^{-L+\tau}\leq 3\eps e^{-L+\tau}<4\eps e^{-L}<\lambda.
	\end{align*}
	This yields $\varphi_{-L}(y)\in B(\varphi_{-L+\tau}(x),\lambda)$ for all $y\in W^u(x,C_i)$
	and hence $\varphi_{-L}W^u(x,C_i)\subset B(\varphi_{-L+\tau}(x),\lambda)$. There exists a $k\in\{1,\dots,n \}$ such that
	\begin{equation*}
		\varphi_{-L}W^u(x,C_i)\subset \varphi_{[-\alpha/2,\alpha/2]}(C_k),
	\end{equation*}
	which is the former. Next, if $y\in W^s(x, C_i)$, then  $y=\langle x,z\rangle_{C_i}$ for some $z\in C_i$ and hence $y= \Gamma g b_s$ for some $s\in[-3\eps,3\eps]$; see Lemma \ref{proexa}.
	The definition of $d_X$ (see \eqref{dX}) and Lemma \ref{at} imply
	\begin{align*}
		d_X(\varphi_{L}(x),\varphi_{L}(y))&=d_X(\Gamma g a_L,\Gamma g b_s a_L)\leq d_\G(g a_L, gb_sa_L)
		\\&<d_\G(a_L,b_sa_L)=d_\G(b_{se^{-L}},e)\leq |s|e^{- L} \leq 3\eps e^{-L}<\lambda.
	\end{align*}
	This yields $\varphi_{L}(y)\in B(\varphi_{L}(x),\lambda)$ for all $y\in W^s(x,C_i)$
	and so $\varphi_LW^s(x,C_i)\subset B(\varphi_L(x),\lambda)$. 
	By the property of $\lambda$, $\varphi_LW^s(x,C_i)\subset \varphi_{[-\alpha/2,\alpha/2]}(C_j)$ for some $j\in \{1,\dots,n\}$,
	owing to Condition (d) in Theorem \ref{step1}. The latter is showed. 	
\end{proof}

The next result is helpful afterwards. 
\begin{lemma}\label{k=1}
	Let $x,y\in C_i$ and $\P_\C(x), \P_\C(y)\in C_j$. If $\P_\C(\langle x,y\rangle_{C_i})\in C_j$, then  $\P_\C(\langle x,y\rangle_{C_i})=\langle\P_\C(x),\P_\C(y)\rangle_{C_j}$.	
\end{lemma}

\begin{proof}
	Let $x,y\in C_i$ and $\P_\C(x), \P_\C(y)\in C_j$ and let $z=\langle x,y\rangle_{C_i}$.
	We first show that if $\P_\C(z)\in C_j$, then $t(x)=t(z)$; recall $t(x)$ and $t(y)$ are the first return times,
	$\P_\C(x)=\varphi_{t(x)}(x)$ and $\P_\C(z)=\varphi_{t(z)}(z)$. This means that
	the first return time is constant along stable manifold. 
	For, write $x=\Gamma g_ic_{u_x}b_{s_x}$ and $\P_\C(x)=\varphi_{t(x)}(x)=\Gamma g_j c_u b_s$ for $|u_x|, |s_x|, |u|, |s| <\eps$, then 
	$\varphi_{t(x)}(z)=\Gamma g_j c_u b_{s+(s_z-s_x)e^{-t(x)}}\in D_j$, which is due to $|{s+(s_z-s_x)e^{-t(x)}}|<3\eps$. 
	On the other hand,  $\P_\C(z)=\varphi_{t(z)}(z)\in C_j\subset D_j$. Since $D_j$ is a local cross section of time $\alpha$,
	and $0<t(x) , t(z)\leq \alpha$, it follows that $\P_\C(z)=\varphi_{t(x)}(z)$ and $t(z)=t(x)$. 
	It remains to show that $\varphi_{t(x)}(\langle x,y\rangle_{C_i})=\langle\varphi_{t(x)}(x),\varphi_{t(y)}(y)\rangle_{C_j}$.
	W.l.o.g, we may assume that $t(y)\leq t(x)$. 
	Let $\tau=t(x)$ and define $s: [0,\tau]\to \R$ by 
	\begin{equation*}
		s(t)=\begin{cases}t&\mbox{if}\quad t\in [0,t(y)],\\
			t(y) &\mbox{if}\quad t\in [t(y),\tau].
		\end{cases}
	\end{equation*} 
	Then $s$ is continuous and $s(0)=0$. Also $\varphi_\tau(x)=\P_\C(x)$ and $\varphi_{s(\tau)}(y)=\P_\C(y)$.
	Furthermore, for $t\in [0,\tau]$,
	\begin{align*}
		d_X(\varphi_t(x),\varphi_{s(t)}(y))
		&\leq d_X(\varphi_t(x),x)+d_X(x,y)+d_X(y,\varphi_{s(t)}(y))\\
		&\leq |t|+|s(t)|+\alpha<3\alpha<\delta_2.  
	\end{align*} 
	Apply Lemma \ref{DD'} to get $\varphi_\tau(\langle x,y\rangle_{C_i})=\langle \varphi_\tau(x),\varphi_{s(\tau)}(y)\rangle_{C_j}$,
	which proves the lemma.
\end{proof}

The next result follows from the previous lemma by induction. 
\begin{lemma}\label{Bk-lemma} Let $K$ be a positive integer and  $x,y\in X$. Suppose that $\P_{\mathscr C}^k(x)\in C_{j_k}, \P_{\mathscr C}^k(y)\in C_{j_k}$ for all $0\leq  k \leq K$.
	If $\P_\C^k(\langle x,y\rangle_{C_{j_0}})\in C_{j_k}$ for all $0\leq k\leq K$,
	then $\P_{\mathscr C}^K(\langle x,y\rangle_{C_{j_0}})=\langle \P_{\mathscr C}^K(x),\P_{\mathscr C}^K(y)\rangle_{C_{j_K}}.$
\end{lemma}

For each $j$, let 
\begin{equation}\label{I_j}
	I_j=\{i:\exists x\in \int C_j\quad\mbox{with}\quad \P_{\mathscr C}(x)\in \int C_i  \}.
\end{equation} 
For $i\in I_j$,  $C_i\cap \varphi_{[0,\alpha]}(C_j)\ne \varnothing$
and hence $B_i\cap \varphi_{[0,\alpha]}(B_j)\ne \varnothing$. By Condition (e') (see Remark \ref{e'}) on the choice of $D_i$'s, we have $C_i\subset  \varphi_{[0,2\alpha]}(D_j)$ and $\pr_{D_j}(C_i)$ makes sense. For $i\in I_j$, define
\begin{equation}
	E_{ji}= C_j\cap \pr_{D_j} (C_i).
\end{equation}
It is clear that $E_{ji}$ is a rectangle having non-empty interior and
we see that
\begin{align*}
	E_{ji}
	& =\{ x\in C_j: x=\pr_{D_j}(y)\ \mbox{ for some }\ y\in C_i \}\\
	& =\{x\in C_j: x=\varphi_\tau(y)\ \mbox{ for some }\ y\in C_i \ \mbox{ and }\  \tau\in [-2\alpha,0]\}\\
	&=\{ x\in C_j: \varphi_{\upsilon}(x)\in C_i \ \mbox{ for some }\ \upsilon\in [0,2\alpha]\}.
\end{align*}

\begin{figure}[h!]
	\begin{center}
		\begin{minipage}{0.5\linewidth}
			\centering
			\includegraphics[angle=0,width=1.2\linewidth]{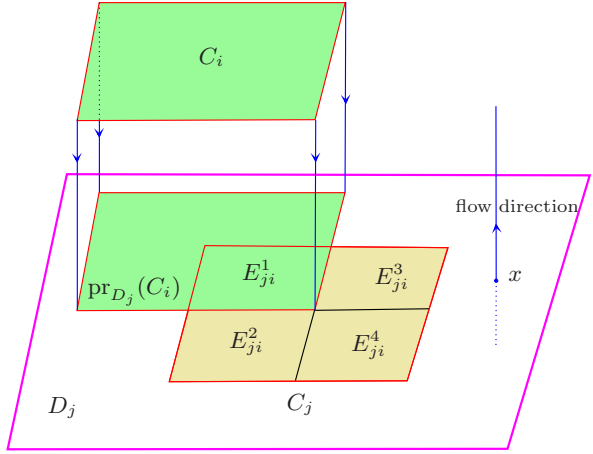}
		\end{minipage}
	\end{center}
	\caption{Projection of $C_i$ on $D_j$ and $E^{1}_{ij},\dots, E^{4}_{ij}$ partition $C_j$}\label{Eij}
\end{figure}

\begin{lemma}\label{Eji1-4}
	Pick $z\in \int E_{ji}$. The sets 
	\begin{align}\label{e1-4}  
		E_{ji}^1&\ =\ \overline {\int E_{ji}},
		\\ \label{e2-4} 
		E_{ji}^2&\ =\ \overline{\{y\in \int C_j:\langle z,y\rangle_{C_j}\in \int E_{ji},
			\langle y,z\rangle_{C_j}\notin E_{ji} \}},
		\\ \label{e3-4} 
		E_{ji}^3&\ = \ \overline{\{y\in \int C_j:\langle z,y\rangle_{C_j}\notin  E_{ji},
			\langle y,z\rangle_{C_j}\in \int E_{ji} \}},
		\\ \label{e4-4} 
		E_{ji}^4&\ =\ \overline{\{y\in \int C_j:\langle z,y\rangle_{C_j}\notin E_{ji},
			\langle y,z\rangle_{C_j}\notin E_{ji} \}}.
	\end{align}
	are rectangles intersecting only
	in their boundaries.
\end{lemma} 
\begin{proof} Since $C_i$ and $C_j$ are rectangles, it follows that 
	$E_{ji}^1$ is a rectangle by  Remark \ref{recrm}. 
	Denote by $G_{ji}^2$ the set under the closure symbol in \eqref{e2-4}.
	For any $y_1,y_2\in G_{ji}^2$, we have $y_1,y_2\in \int C_j$ and hence $\langle y_1,y_2\rangle \in \int C_j$. 
	Furthermore, $\langle z,y\rangle_{C_j}
	=\langle z,\langle y_1,y_2\rangle_{C_j}\rangle_{C_j}
	=\langle z,y_2\rangle_{C_j}	\in \int E_{ji}$, owing to $y_2\in G_{ij}^2$.
	Also,  $\langle y,z\rangle_{C_j}=\langle\langle y_1,y_2\rangle_{C_j},z\rangle_{C_j}=\langle y_1,z\rangle_{C_j}\notin \int E_{ji}$
	due to $y_1\in G_{ij}^2$. Therefore $y=\langle y_1,y_2\rangle_{C_j}\in G_{ji}^2$.
	Since $\langle \cdot, \cdot \rangle_{C_j}$ is continuous on $C_j\times C_j$, we deduce that  $E_{ij}^2$ is a rectangle. 
	Analogously, $E_{ij}^3,E_{ij}^4$ are rectangles. 
	
	In addition, 
	\[C_j=E_{ji}^1\cup E_{ji}^2\cup E_{ji}^3\cup E_{ji}^4.\]
	As $\int E_{ji}, G_{ji}^2, G_{ji}^3,G_{ji}^4$ are pairwise disjoint, 
	$E_{ji}^1, E_{ji}^2, E_{ji}^3, E_{ji}^4$ intersect only in their boundaries;
	See Figure \ref{Eij} for an illustration.
\end{proof}
\begin{lemma}\label{frankC}
	The sets
	$F_j^{a(i)}:=\overline{\bigcap_{i\in I_j} \int{E_{ji}^{a(i)}}}, a(i):I_j\to \{1,2,3,4\}$
	are rectangles and create  a cover of $C_j$. 
	Furthermore,
	elements in \[{\mathfrak C}_j=\{F_j^{a(i)}, a(i):I_j\to\{1,2,3,4 \}\}\] intersect only in their boundary,
	and \[U_j=\bigcup_{E\in {\mathfrak C_j}}{\int E} \] 
	is an open dense subset of $C_j$. 
\end{lemma}
\begin{proof}
	By Lemma \ref{Eji1-4}, $E_{ji}^{a(i)}$ are rectangles, so are $F_j^{a(i)}$
	by Remark \ref{recrm} and it is clear that they are a cover of $C_j$. The last assertion is obvious.  
\end{proof} 

\begin{figure}[h!]
	\begin{center}
		\begin{minipage}{0.7\linewidth}
			\centering
			\includegraphics[angle=0,width=1.2\linewidth]{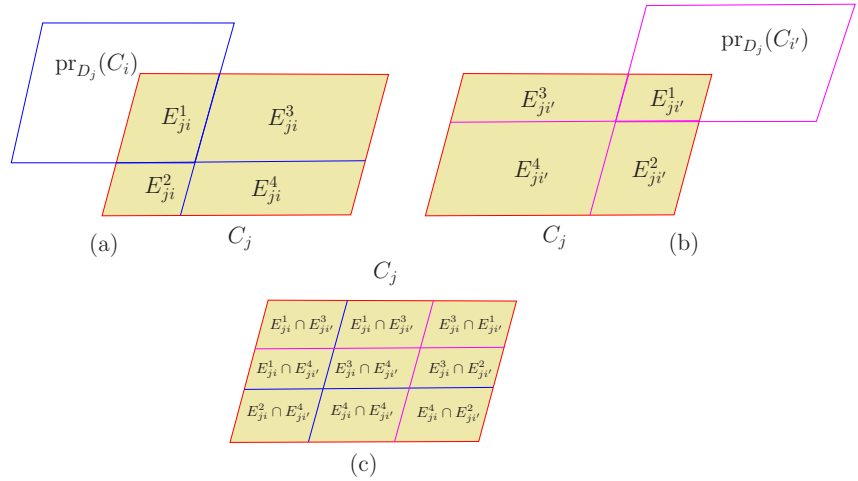}
		\end{minipage}
	\end{center}
	\caption{For $I_j=\{i,i'\}$: projections of $C_i$ and $C_{i'}$ to $D_j$ create  partitions of $C_j$ in (a) and (b);  the set  ${\mathfrak{C}_j}$ consists of nine sets in (c). }\label{Eji}
\end{figure}

Denote by $\P_{\mathscr C}$ the Poincar\'e map for proper family ${\mathscr C}=\{C_1,\dots,C_n\}$.
%\[{\tt C}_\Z=\big\{x\in {\tt C }:  \P_{\mathscr C}^k(x)\in \bigcup_{j=1}^n U_j\quad\mbox{for all}
%\quad k\in \Z \big\}. \]
For  a positive integer $N$, we define
\begin{equation}\label{CN}
	{\cal C}_N=\big\{x\in C_1\cup\dots\cup C_n:  \P_{\mathscr C}^k(x)\in \bigcup_{j=1}^n U_j\quad\mbox{for all}
	\quad k=1,\dots,N \big\}
\end{equation}
and an equivalence relation on ${\cal C}_N$ as follows. For $x,y\in {\cal C}_N$, the relation
$x\overset{N}\sim y$ means that, for every $k\in\{1,\dots, N\}$,
$\P_{\mathscr C}^k(x)$
and $\P_{\mathscr C}^k(y)$ not only lie the same $C_{j_k}\in{\mathscr C}$ but also the same member 
$F^{a(i)}_{j_k}$ of ${\mathfrak C}_{j_k}$ for some $j_k\in \{1,\dots,n\}$. 
Let $G_1,\dots, G_m$ denote the equivalence classes. Since $N$ is finite and there are finitely many $C_j$ and finitely many members in ${\mathfrak C}_j$, it follows that $m=m(N)$ is finite.

\begin{lemma}\label{reclmG} The sets $\overline G_1,\dots, \overline G_m$ are rectangles in $X$.
\end{lemma}
\begin{proof} We follow the proof of Lemma 7.5 in \cite{bo-symb}. 
	For $p\in\{1,\dots,m\}$ fixed, it is enough to verify that
	if $x,y\in C_{j_0}, x,y\in G_p$, then $z=\langle x,y\rangle_{C_{j_0}}\in G_p$.  Since $G_p$ is an equivalence class,
	in order to achieve $z\in {G}_p$, we must show that $x\overset{N}\sim z$ or $y\overset{N}\sim z$. 
	This means that for each $k\in \{1,\dots,N\}$, $\P_{\mathscr C}^k(x), \P_{\mathscr C}^k(y)$
	and $\P_{\mathscr C}^k(z)$  belong to the same $C_{j_k}$
	for some $j_k\in \{1,\dots,n\}$ and the same member of ${\mathfrak C}_{j_k}$. 
	This is clear for $k=0$ since $C_{j_0}$ and $F_{j_0}^{a(i)}$ are rectangles.
	Suppose on the contrary that it is true for all $0\leq k<k'$
	but not for some $k'\leq N$.  
	Then $\P_{\mathscr C}^{k'-1}(x),\,\P_{\mathscr C}^{k'-1}(y)$
	and $\P_{\mathscr C}^{k'-1}(z)$ all lie in some $C_{j}$;
	$\P_{\mathscr C}^{k'}(x),\,\P_{\mathscr C}^{k'}(y)$ lie in some $C_{i'}$ but
	$\P_{\mathscr C}^{k'}(z)$ lies in a $C_{i}\ne C_{i'}$. 
	Then by the definition of $I_{j}$ (see \eqref{I_j}), $i',i\in I_{j}$ and  $i'\ne i$.
	It follows from Lemma \ref{Bk-lemma} that 
	\begin{equation}\label{k'-1} 
		\langle \P_{\mathscr C}^{k'-1}(x),\P_{\mathscr C}^{k'-1}(y) \rangle_{C_{j}}
		=\P_{\mathscr C}^{k'-1}(\langle x,y\rangle_{C_{j_0}})=\P_{\mathscr C}^{k'-1}(z). 
	\end{equation} 
	Recall that
	$\P_{\mathscr C}^{k'-1}(x)$ and $\P_{\mathscr C}^{k'-1}(y)$ lie in the same member of ${\mathfrak C}_{j}$
	and each member of ${\mathfrak C}_{j}$ is a rectangle. 
	It follows from \eqref{k'-1} that $\P_{\mathscr C}^{k'-1}(z)$ lies in that member too. 
	Note that  $z':=\P_{\mathscr C}^{k'-1}(z)\in C_{j}\cap \pr_{D_{j}}(C_{i})=E_{j{i}}$, which is due to 
	$z'\in C_{j}$ and $\P_{\mathscr C}(z')\in C_{i}$. Then $z'\in E_{ji}^1$ implies that  
	$x'=\P_{\mathscr C}^{k'-1}(x)\in E_{ji}^1$, owing to that both $x'$ and $z'$ lie in the same member of ${\mathfrak C}_{j}$. This yields 	 $\varphi_\tau(x')\in C_{i}$ for some $0<\tau\leq 2\alpha$. 
	Since $\P_{\mathscr C}(x')=\P_{\mathscr C}^{k'}(x)\in C_{i'}$, there is an 
	$s$ with $0<s< \tau$ so that $\varphi_s(x')\in C_{i'}$, and hence \begin{equation}\label{eij}\varphi_{\tau}(x')=\varphi_{\tau-s}(\varphi_{s}(x'))\in C_{i}\cap \varphi_{[0,2\alpha]} (C_{i'})
		\subset D_{i}\cap \varphi_{[0,2\alpha]}(D_{i'}).
	\end{equation}	On the other hand, $x'\in E_{{j}i'}^1$ yields  $z'\in E_{ji'}^1$. 
	There is $0<\tau'\leq 2\alpha$ such that $\varphi_{\tau'}(z')\in C_{i'}$.
	Since $\P_\C(z')\in C_i$, it follows that $\varphi_{s'}(z')\in C_i$ for some $0<s'<\tau'$.
	As a consequence,
	\[\varphi_{\tau'}(z')=\varphi_{\tau'-s'}(\varphi_{s'}(z'))\in C_{i'}\cap \varphi_{[0,2\alpha]}(C_i)
	\subset D_{i'}\cap \varphi_{[0,2\alpha]}(D_i),\]
	which is impossible due to \eqref{eij} and  Condition (c) in Theorem \ref{step1}. Therefore, $\P_{\mathscr C}^k(z)$ lies in the same $C_{j_k}$ 
	as $\P_{\mathscr C}^k(x)$ and $\P_{\mathscr C}^k(y)$ for $0\leq k\leq N$. In addition,
	it follows from Lemma \ref{Bk-lemma} that 
	\[\P_{\mathscr C}^k(z)=\langle \P_{\mathscr C}^k(x),\P_{\mathscr C}^k(y)\rangle_{C_{j_{k}}}.\]
	Since for each $k\in\{1,\dots,n\}$, $\P_{\mathscr C}^k(x)$ and $\P_{\mathscr C}^k(y)$ 
	lie in the same member of ${\mathfrak C}_j$ and each member is a rectangle, $\P_{\mathscr C}^k(z)$ 
	must lie in that member too.
	We have shown that if $x,y\in G_p$, then $\langle x,y\rangle\in G_p$, which implies that  
	$\overline {G}_p$ is a rectangle.
	
\end{proof}

Let $\tau_1,\dots, \tau_m$ be so small distinct numbers that  $\varphi_{\tau_1}(\overline G_1),
\dots, \varphi_{\tau_m}(\overline G_m)$ are pairwise disjoint.
Using Lemma \ref{reclmG},   $M_p:=\varphi_{\tau_p}(\overline G_p),p=1,\dots,m$
are rectangles.   We are going to show that 
\[{\mathscr M}_N=\{M_1,\dots,M_{m(N)}\}  \]
is a Markov partition. 

For any $p\in \{1,\dots,N\}$, there is $i=i(q)\in \{1,\dots,n\}$ such that
$G_p\subset C_i\subset  D_i$. Write $\widehat D_p=\varphi_{\tau_p}(D_i)$ to have 
$M_p\subset \widehat D_p$. By Condition (c) in Theorem \ref{step1}, if $i\ne j$, then 
at least one of the sets
$\widehat D_i\cap \varphi_{[0,2\alpha]}(\widehat D_j)$ and $\widehat D_j\cap \varphi_{[0,2\alpha]}(\widehat D_i)$
is empty. Furthermore, $X=\bigcup_{i=1}^m \varphi_{[-\alpha,0]}(C_i)$
implies that $X=\bigcup_{j=1}^m\varphi_{[-2\alpha,0]}(M_j)$. It follows that  ${\mathscr M}_N$ is a proper family of size $2\alpha$.

The final lemma below proves Theorem \ref{mthm}.
\begin{lemma}\label{flm2}
	For  $N>\frac{L}{2\alpha}$, ${\mathscr M}_N$
	is a Markov partition of time $2\alpha$.
\end{lemma}
\begin{proof} 
	We only need to show that  ${\mathscr M}_N$ satisfies the Markov property; see Definition
	\ref{MPdn}. Denote by $\P_{\mathscr M}$ the corresponding return map of   ${\mathscr M}_N$
	and
	\[{\cal M}_N^*=\{ x\in M_1\cup\cdots \cup M_N\, :\, 
	\P_{\mathscr M}^k(x)\in \int M_1\cup\dots\cup \int M_N\ \mbox{for all}\ k\in\Z \}.\]
	We only prove $(M_s)$. 
	Recall
	\[U(M_p,M_q)=\overline{\{  z\in {\mathcal M}_N^*, z\in M_p, \P_{\mathscr M}(z)\in M_q    \}}.  \]
	We must show that $W^s(x',M_p)\subset U(M_p,M_q)$ for $x'\in U(M_p,M_q)$. 
	Since  $U(M_p,M_q)$ is closed, $U(M_p,M_q)\cap {\cal M}_N^*$
	is dense in $U(M_p,M_q)$, 
	and $W^s(x',M_p)$ varies continuously with $x'$, 
	it is enough to show the inclusion  for $x'\in  U(M_p,M_q)\cap {\mathcal  M}^*_N$. 
	Also, due to $W^s(x',M_p)\cap \varphi_{\tau_p}(G_p)$ is dense in $W^s(x',M_p)$,
	it remains to show  $y'\in U(M_p,M_q)$ for $y'=\varphi_{\tau_p}(y)$ with $y\in W^s(x,\overline G_p)\cap G_p$; 
	here $x'=\varphi_{\tau_{p}}(x)$. 
	It is enough to show that $\P_\C(y)\in W^s(\P_\C(x),G_q)$.
	
	Let \[ \P_{\mathscr C}^k(x)\in C_{j_k}\quad \mbox{for}\quad 1\leq k\leq N. \]
	Since $x\overset{N}\sim y$, $x_k:=\P_\C^k(x)$ and $y_k:=\P_\C^k(y)$
	are in the same $C_{j_k}$ and in the same member of ${\mathfrak C}_{j_k}$ for all $0\leq k\leq N$.  
	Due to $y\in W^s(x, G_p)$, $y=\langle x,z\rangle_{G_p}$ for some $z\in G_p$.
	We have $x,y,z\in C_{j_0}$ and $\P_\C(x),\P_\C(y), \P_\C(z)\in C_{j_1}$.
	Apply Lemma \ref{k=1} to have $\P_\C(y)=\langle \P_\C(x),\P_\C(z)\rangle_{C_{j_1}}$
	and also $\P_\C(y)\in W^s(\P_\C(x),C_{j_1})$.
	In order to achieve $\P_\C(y)\in W^s(\P_\C(x),G_q)$, we must show $\P_\C(y)\in G_q$,
	or equivalently,  $\P_\C(y)\overset{N}\sim \P_\C(x)$.
	Owing to the fact that $y\overset{N}\sim x$, 
	it remains to show $\P_\C^{N+1}(x)$ and $\P_\C^{N+1}(y)$ are in the same $C_{j_{N+1}}$
	and the same member of ${\mathfrak C}_{j_{N+1}}$.
	
	Let \[x_{N+1}=\P_{\mathscr C}^{N+1}(x)\in C_{j_{N+1}} \quad\mbox{and}\quad y_{N+1}=\pr_{D_{j_{N+1}}}(\P_\C^N(y))\in D_{j_{N+1}}.
	\]
	We first claim that
	\begin{equation}\label{y1}
		y_{N+1}\in W^s(x_{N+1},C_{j_{N+1}}).
	\end{equation}
	By Lemma \ref{CiCj}\,(a), there is $l\in\{1,\dots,n \}$ such that $\varphi_{-L}(x_{N+1})\in\varphi_{[-\alpha/2,\alpha/2]} (C_l)$
	and
	\[\pr_{D_{j_{N+1}}}\varphi_{L}W^s(\pr_{D_l}\varphi_{-L}(x_{N+1}),C_l)
	\subset W^s(x_{N+1},C_{j_{N+1}}). \]
	For $N>\frac{L}{2\alpha}$, there is a $k\in \{ 0,\dots,N\}$ such that
	\[\pr_{D_l}\varphi_{-L}(x_{N+1})=\P_{\mathscr C}^{-k}(x_N)=\P_\C^{N-k}(x). \] 
	Furthermore, using $x\overset{N}\sim y$
	and $y\in W^s(x, C_{j_0})$, it follows from Lemma \ref{Bk-lemma} that
	$\P_{\mathscr C}^{-k}(y_N)\in W^s(\P_{\mathscr C}^{-k}(x_N), C_l)$, and also   
	\[y_{N+1}=\pr_{D_{j_{N+1}}}\varphi_{L}(\P_{\mathscr C}^{-k}(y_N))\in W^s(x_{N+1},C_{j_{N+1}}),\]
	which is \eqref{y1}. As a result, $y_{N+1}\in C_{j_{N+1}}$.

	Next, we check that $x_{N+1},y_{N+1}$ are in the same member of ${\mathfrak C}_{j_{N+1}}$ and
	also
	$y_{N+1}=\P_{\mathscr C}(y_N)$. We first suppose that $x_{N+1},y_{N+1}$ are not in the same member of ${\mathfrak C}_{j_{N+1}}$. Then  there is some $i\in I_{j_{N+1}}$
	for which $x_{N+1}$ and $y_{N+1}$ are in different ${E_{j_{N+1}i}^a}$'s; see Figure \ref{Eji}. 
	Taking $z'\in \int E_{j_{N+1}i}$, due to \eqref{y1},  
	$\langle y_{N+1},z'\rangle_{C_{j_1}}=
	\langle  x_{N+1},z'\rangle_{C_{j,N+1}}$ by Lemma \ref{=x=ylm}. Since $x_{N+1}$ and $y_{N+1}$
	are in different $E_{j_{N+1}}^a$'s, we may assume that 
	$\hat x:=\langle z',x_{N+1} \rangle_{C_{j_{N+1}}}\in E_{j_{N+1}i}$
	and $\hat y:=\langle z',y_{N+1}\rangle_{C_{j_{N+1}}}\notin E_{j_{N+1}i}$; see \eqref{e1-4}-\eqref{e4-4}.
	Then $\bar x=\pr_{D_{i}}(\hat x)\in C_i$ and $\bar y=\pr_{D_i}(\hat y)\notin C_i$. 
	Let $y_{N+1}'=\pr_{D_{i}}(y_{N+1})$ and $x_{N+1}'=\pr_{D_{i}}(x_{N+1})$.
	Using Lemma \ref{CiCj}\,(a), there is $s\in\{1,\dots,n\}$ so that
	$\varphi_{-L} (\bar x)\in \varphi_{[-\alpha/2,\alpha/2]} (C_s)$
	and
	\begin{equation}\label{wu}  \pr_{D_i}\varphi_{L} W^s(\pr_{D_s}\varphi_{-L}(\bar x),C_s)\subset W^s(\pr_{D_i}(\bar x),C_i).
	\end{equation}  
	Furthermore, due to $N>\frac{L}{2\alpha}$, there is $k'\in \{0,\dots, N\}$ such that $\P_\C^{N-k'}(x)=\P_{\mathscr C}^{-k'}(x_{N+1})=\pr_{D_s}\varphi_{-L}(x_{N+1}')\in C_s$. 
	Also  $\P_\C^{N-k'}(y)=\P_{\mathscr C}^{-k'}(y_N)=\P_{\mathscr C}^{-k'}(y_N)\in C_s$
	since $y\overset{_{N}}{\sim} x$. 
	Owing to $$\hat y=\langle z', y_{N+1}\rangle_{D_{j_{N+1}}}
	=\langle\langle z', x_{N+1}\rangle_{D_{j_{N+1}}},y_{N+1}\rangle_{D_{j_{N+1}}}
	= \langle \hat x,y_{N+1}\rangle_{D_{j_{N+1}}},$$ 
	we get
	$$\bar y=\pr_{D_i}(\hat y)=\langle \pr_{D_i}(\hat x),\pr_{D_i}(y_{N+1})\rangle_{D_i}=\langle \bar x,y'_{N+1}\rangle_{D_i}$$ by Lemma \ref{DD'}. Also, apply Lemma \ref{DD'} again to obtain
	\begin{align*}
		\pr_{D_s}\varphi_{-L}(\bar y)=\ &\langle \pr_{D_s}\varphi_{-L}(\bar x),\pr_{D_s}\varphi_{-L}(y'_{N+1})\rangle_{D_s}
		\\
		=\ &\langle \pr_{D_s}\varphi_{-L}(\bar x),\P_{\mathscr C}^{-k'}(y_N)\rangle_{C_s}
		\in W^s(\pr_{D_s}\varphi_{-L}(\bar x),C_s).
	\end{align*}
	It follows from \eqref{wu} that 
	\[\bar y=\pr_{D_i}\varphi_{L}(\pr_{D_s}\varphi_{-L}(\bar y))\in W^s(\bar x,C_i)\subset C_i,\]
	which contradicts $\bar y\notin C_i$ and hence $x_{N+1}, y_{N+1}$ are in the same member of ${\mathfrak C}_{j_{N+1}}$.
	
	Next, we verify that $y_{N+1}=\P_\C(y_N)$. Suppose on the contrary that $\P_{\mathscr{C}}(y_N)=\bar y_{N+1}\in C_k$ for some $k\ne j_{N+1}$.
	Then $y_N\in E_{j_Nk}$ implies that $x_N\in E_{j_Nk}$ since $x_N, y_N$ belong to the same member of
	${\mathfrak C}_{j_N}$. There is $\tau\in(t(x_N),2\alpha]$ such that 
	$\varphi_\tau(x_N)\in C_k$. This implies 
	\begin{equation}\label{eql}
		\varphi_\tau(x_N)=\varphi_{\tau-t(x_N)}(x_{N+1})\in C_k\cap \varphi_{[0,2\alpha]}(C_{j_{N+1}})\subset D_k\cap \varphi_{[0,2\alpha]}(D_{j_{N+1}}).
	\end{equation} 
	On the other hand, $y_{N+1}=\varphi_s(y_N)\in D_{j_{N+1}}$ for some $s\in(0, 2\alpha]$
	and $\bar y_{N+1}=\varphi_{t(y_N)}(y_N)\in C_k$ with $0<t(y_N)<s$.
	Then 
	\begin{equation}
		y_{N+1}=\varphi_{s}(y_N)=\varphi_{s-t(y_N)}(\bar y_{N+1})\in 
		D_{j_{N+1}}\cap \varphi_{[0, 2\alpha]}(D_k),
	\end{equation}
	which is impossible due to \eqref{eql} and Condition (e) in Theorem \ref{step1}. Therefore, $y_{N+1}=\P_{\mathscr C}(y_N)$ and so $y_1\overset{N}\sim x_1$. 
	To summarize, we have proved that $\P_{\mathscr C}(y)\in G_q$ as well as $y_1\in W^s(x_1, G_q)$. This completes
	the proof of $(M_s)$. The proof of $(M_u)$ is analogous.

\end{proof}

\noindent {\bf Acknowledgments:} The author thanks Prof. Markus Kunze for hospitality during his stay at Cologne University.
A part of this work was done while the author was working at Vietnam Institute for Advanced Study in Mathematics (VIASM).
 He would like to thank VIASM for its wonderful working condition. This work is supported by Vietnam’s National Foundation for Science and Technology Development (Grant No. 101.02-2020.21).

%\section{Comments}

\end{document}